\documentclass{amsart}

    \usepackage[utf8]{inputenc}
    \usepackage{amsfonts}
    \usepackage{amsmath}
    \usepackage{amsthm}
    \usepackage{amssymb}
    \usepackage{mathtools} 
    \usepackage{stmaryrd} 
    \usepackage{bm} 
    \usepackage{enumitem} 
    \usepackage[style=alphabetic,maxbibnames=99]{biblatex}
        \DeclareFieldFormat{title}{\mkbibquote{#1}}
    \usepackage{hyperref}
    \usepackage[noabbrev,nameinlink,capitalise]{cleveref}
        \crefname{subsection}{Subsection}{Subsections}
        \crefname{subsection}{Subsection}{Subsections}
    \usepackage{subcaption}
    \usepackage{color}
    \usepackage[dvipsnames]{xcolor}
    \usepackage{graphicx}\graphicspath{{./images/}}
    \usepackage{tikz}
        \usetikzlibrary{calc}
        \usetikzlibrary{arrows}
        \usetikzlibrary{decorations.pathmorphing} 
        \tikzset{every picture/.style = {line width=1pt}}
        \tikzset{vertex/.style = {shape=circle,draw,inner sep=1pt, outer sep=1pt}}
        \tikzset{edge/.style = {->,> = latex'}}
        \tikzset{curvy/.style={decorate, decoration=snake,segment length=.25cm}}

    \theoremstyle{plain}
        \newtheorem{theorem}{Theorem}[section]
        \newtheorem*{maintheorem}{Main Theorem}
        \newtheorem{corollary}[theorem]{Corollary}
        \newtheorem{lemma}[theorem]{Lemma}
        \newtheorem{proposition}[theorem]{Proposition}
        \newtheorem{claim}[theorem]{Claim}
    \theoremstyle{definition}
        \newtheorem{definition}[theorem]{Definition}
        
    \theoremstyle{remark}
        \newtheorem{remark}[theorem]{Remark}
        \newtheorem{question}[theorem]{Question}
    \newenvironment{sketch}{%
        \proof}{\endproof}
    \newtheoremstyle{named}{}{}{\itshape}{}{\bfseries}{.}{.5em}{\thmnote{#3's }#1} 
        \theoremstyle{named}

    \title{On Thompson Groups for Wa\.zewski Dendrites}
    \author{Matteo Tarocchi}
    \date{}
    
    \thanks{
    The author is supported by the French project GoFR (ANR-22-CE40-0004).
    The author is a member of the Gruppo Nazionale per le Strutture Algebriche, Geometriche e le loro Applicazioni (GNSAGA) of the Istituto Nazionale di Alta Matematica (INdAM)
    }
    
    \address{Dipartimento di Matematica e Applicazioni, Universit\`a degli Studi di Milano-Bicocca, Milano, Italy, EU.
    Current: Université Paris-Saclay, CNRS, Laboratoire de mathématiques d’Orsay, Orsay, France, EU \& Université de Rennes, CNRS, IRMAR, Rennes, France, EU
    }
    \email{\href{mailto:matteo.tarocchi.math@gmail.com}{matteo.tarocchi.math@gmail.com}}

\begin{document}

\begin{abstract}
    We study a family of Thompson-like groups built as rearrangement groups of fractals from \cite{belk2016rearrangement}, each acting on a Wa\.zewski dendrite.
    Each of these is a finitely generated group that is dense in the full group of homeomorphisms of the dendrite (studied in \cite{DM19}) and has infinite-index finitely generated simple commutator subgroup, with a single possible exception.
    More properties are discussed, including finite subgroups, the conjugacy problem, invariable generation and existence of free subgroups.
    We discuss many possible generalizations, among which we find the Airplane rearrangement group $T_A$.
    Despite close connections with Thompson's group $F$, dendrite rearrangement groups seem to share many features with Thompson's group $V$.
\end{abstract}



\maketitle


\section{Introduction}

The main trio of Thompson groups, commonly denoted by $F$, $T$ and $V$, was introduced in the 60's by Richard Thompson.
The two bigger siblings $T$ and $V$ became the first known examples of infinite finitely presented simple groups, while $F$, despite many decades of attempts, is still shrouded in the mystery that concerns its amenability.
These are countable groups of ``piecewise-rigid'' homeomorphisms of the unit interval, the unit circle and the Cantor space, respectively:
the main idea behind this ``piecewise-rigidity'' is that the space on which they act is equipped with dyadic subdivision rules that provide arbitrarily granular finite partitions, and the elements of a Thompson's group act by homeomorphisms that are also bijections between the pieces of two such partitions.
A common reference for an introduction to these groups is \cite{cfp}.
Thompson groups have prompted many generalizations based on the same idea of permuting the pieces of objects of some sort of a rewriting system, among which we recall, as examples:
the Thompson-Stein groups from \cite{Stein}, where multiple arities of subdivisions are allowed;
the diagram groups studied in \cite{guba1997diagram}, where the subdivision rules are relations of a semigroup presentation;
the higher-dimensional versions of $V$ introduced in \cite{Brin2004HigherDT}, based on subdivisions of Cantor cubes, and in \cite{TwistednV} a version of these last groups that allows twisting the pieces of the partition.

Another notable example is the family of rearrangement groups of fractals recently introduced in \cite{belk2016rearrangement} by J. Belk and B. Forrest.
These are groups of ``piecewise-rigid'' homeomorphisms of certain fractals that are built using subdivision rules based on finite graphs.
The family of rearrangement groups includes the original trio of Thompson groups $F$, $T$ and $V$ along with their natural generalizations $F_{n,k}$, $T_{n,k}$ and $V_{n,k}$ (commonly known as Higman-Thompson groups), and it features new Thompson-like groups such as the Basilica rearrangement group $T_B$, introduced earlier in \cite{Belk_2015} as a prototype of rearrangement groups, and the Airplane rearrangement group $T_A$, introduced in \cite{belk2016rearrangement} as an example and then studied in \cite{Airplane}.
It also includes certain diagram groups (Example 2.14 of \cite{belk2016rearrangement}), along with the Houghton groups $H_n$ from \cite{Houghton1978TheFC} and the Thompson-like groups $QF$, $QT$ and $QV$ introduced in \cite{QV} and also studied in \cite{QV1} (see \cite{RearrConj} for a construction of $QF$, $QT$, $QV$ and the Houghton groups as rearrangement groups).
Rearrangement groups can be embedded in Thompson's group $V$ (see \cref{prop:dendrite:Thompson:comparison} and \cite[Remark 1.8]{RearrConj}), so in particular they are countable groups.

Since this family has been introduced only recently, the list of results that are known about them is still fairly short:
in \cite{TBnotfinpres} it was showed that the Basilica rearrangement group $T_B$ is not finitely presented;
\cite{IG} provides a sufficient condition for non-invariable generation;
in \cite{BasilicaDense} it is proved that $T_B$ is dense in the group of all orientation-preserving homeomorphism of the Basilica fractal;
\cite{RearrConj} studies the conjugacy problem in this family of groups.

Among the many fractals on which a rearrangement group can act, dendrites (connected, locally connected, compact metric spaces without simple closed curves, such as the one in \cref{fig:D3}) have already inspired multiple works from group theorists, arguably because of their tree-like structure (the sense in which they are tree-like can actually be formalized: see Definition 2.12 and Theorem 10.32 of \cite{continua}).
For example, \cite{UniversalDendrites} is about the group of homeomorphisms of universal dendrites (i.e., dendrites that contain a homeomorphic copy of every dendrite), \cite{AmenableDendrites} is about actions of amenable countable groups on dendrites and \cite{DM18} is about the dynamics of group actions on dendrites.
The recent work \cite{DM19}, which is going to be a fundamental reference here, provides a rich study of the group of homeomorphisms of a dendrite.
Moreover, \cite{Kaleid} uses Wa\.zewski dendrites to build simple permutation groups with interesting dynamical and topological properties, and \cite{Duc20} provides a deep study of the group of homeomorphisms of the infinite-order Wa\.zewski dendrite $D_\infty$.

\begin{figure}
\centering
\includegraphics[width=.667\textwidth]{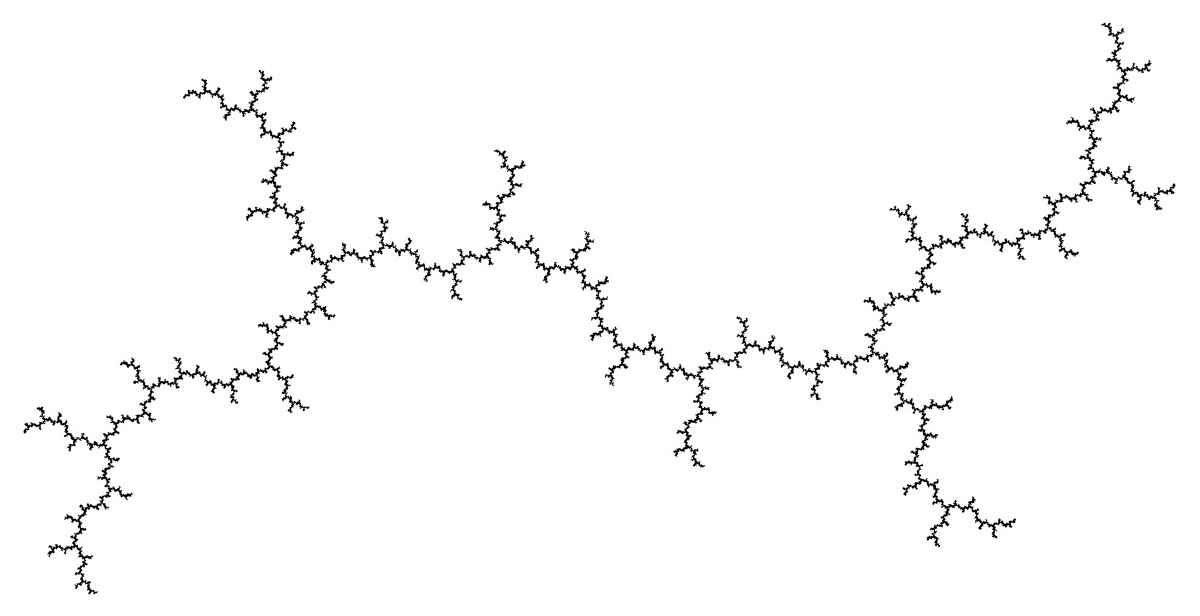}
\caption{The Wa\.zewski dendrite $D_3$ (\href{https://commons.wikimedia.org/wiki/File:Julia_IIM_1.jpg}{image} by Adam Majewski, licensed under \href{https://creativecommons.org/licenses/by/3.0}{CC BY 3.0}).}
\label{fig:D3}
\end{figure}

This work belongs to the intersection of these two fields of study in group theory, that of Thompson-like groups and that of groups acting on dendrites.
We will make use of strategies from both worlds, without assuming that the reader is familiar with any of them.
We will consider Wa\.zewski dendrites $D_n$ ($n \geq 3$, \cref{fig:D3} depicts $D_3$), which can be thought as ``dense $n$-regular trees'', and we will study rearrangement groups of Thompson-like homeomorphisms of them.
After providing formal definitions and the necessary background on Wa\.zewski dendrites and rearrangement groups in \cref{sec:background}, in \cref{sec:dendrite:rearrangement:groups} we will exhibit the replacement systems $\mathcal{D}_n$ (which is the aforementioned subdivision rules) that realize Wa\.zewski dendrites as limit spaces, and their rearrangement groups $G_n$ are the subjects of the subsequent \cref{sec:gen,,sec:dns,,sec:comm,,sec:properties}, where we prove the results collected in the following statement:

\begin{maintheorem}
For every $n\geq4$, the dendrite rearrangement group $G_n$ is finitely generated by a copy of Thompson's group $F$ and a copy of the permutation group $\mathrm{Sym}(n)$ (\cref{sec:gen}) and it is dense in the full group of homeomorphisms of the Wa\.zewski dendrite $D_n$ (\cref{sec:dns}).

The abelianization $G_n / [G_n, G_n]$ of $G_n$ is isomorphic to the infinite group $\mathbb{Z}_2 \oplus \mathbb{Z}$, and the commutator subgroup $[G_n, G_n]$ of $G_n$ is is finitely generated and, if $n \geq 4$, then it is simple (\cref{sec:comm}).

Dendrite rearrangement groups embed into $V$, but they do not embed into $T$ nor does $T$ embed into any of them.
Moreover, they are not invariably generated, their conjugacy problem is solvable, they contain non-abelian free subgroups, they are pairwise non-isomorphic and the cardinality of their finite subgroups have been classified (\cref{sec:properties}).
\end{maintheorem}

It is not yet known whether or not the commutator subgroup of $G_3$ is simple, as will be discussed in \cref{rmk:3:comm:simple}.

Finally, \cref{sec:generalizations} explores the relation between dendrite rearrangement groups and the Airplane rearrangement group $T_A$ studied in \cite{Airplane}, proving that it is dense in the group of all orientation-preserving homeomorphisms of the Wa\.zewski dendrite $D_\infty$, and hints at possible generalizations of these groups, such as the rearrangement groups of generalized Wa\.zewski dendrites $D_S$ for finite subsets $S \subset \mathbb{N}_{\geq3}$.

We believe that results about density of rearrangement groups in their ``ambient'' overgroup (which are not uncommon, as discussed in \cref{sub:dns}) could produce interesting connections with computer science.
These results could provide a tool for ``approximating'' uncountable homeomorphism groups of fractals by Thompson-like groups, which, being countable and showing nice behaviour with regards to decision problems (such as often having solvable conjugacy problem by \cite{RearrConj}), are arguably quite suitable for computer processing.

Overall, dendrite rearrangement groups seem to have multiple features in common with Thompson's groups $F$:
they are generated by a copy of $F$ along with a finite group, hence they inherit many transitivity properties from $F$, and the commutator subgroup is partly characterized by the behaviour at the endpoints.
They also share features with Thompson's group $V$, such as featuring torsion, not being invariably generated and including copies of non-abelian free subgroups.
And, in some way, they resemble neither $F$ nor $V$.


\section{Background on Dendrites and Rearrangement Groups}
\label{sec:background}

This Section provides a quick introduction to the topics of dendrites and rearrangement groups.
In particular, the information needed about Wa\.zewski dendrites will be given in \cref{sub:dendrites}, while \cref{sub:rep:sys,sub:limit:spaces,sub:rearrangements} describe what a \textit{replacement system} is, how it produces a fractal which we will call its \textit{limit space} and how it prompts the definition of a group of piecewise-canonical homeomorphisms of the limit space, called \textit{rearrangements}.

It is understood that these Subsections only provide brief overviews.
Chapter X of \cite{continua} gives further general details on dendrites, whereas \cite{belk2016rearrangement} should be consulted for a fully detailed introduction to replacement systems, limit spaces and rearrangement groups.

\subsection{Wa\.zewski Dendrites}
\label{sub:dendrites}

We begin with the general definition of dendrites, then we introduce a concept of order of a point which will be useful to define what Wa\.zewski dendrites are.

\begin{definition}
\label{def:dendrite}
A \textbf{dendrite} is a non-empty, locally connected, compact, connected metric space that contains no simple closed curves.
A dendrite is \textbf{degenerate} if it is a singleton.
A \textbf{subdendrite} is
a dendrite that is the subspace 
of a dendrite.
\end{definition}

There are many equivalent definitions of dendrites.
For example, we could also define a dendrite to be a non-empty compact connected metric space such that the intersection of any two connected subsets is a connected subset (\cite[Theorem 10.10]{continua}).
See also (1.1) and (1.2) in Section V.1 of \cite{AnalyticTopology} for further characterizations of dendrites among continua and locally connected continua.
Non-degenerate dendrites are also the underlying topological spaces behind the so-called \textit{metric trees} (\cite[Proposition 2.2]{MetricTrees}).
Examples of dendrites are the compactification of any simplicial tree (as noted in \cite[Proposition 12.2]{DM18}) and certain Julia sets (for instance, for the complex map $z \to z^2 + i$).

The \textbf{order} of a point $p$ of a dendrite $X$ is the cardinality of the set of connected components of $X \setminus \{p\}$.
(For dendrites, this is equivalent to the more general notion of Menger–Urysohn order, which is discussed in \cite[7, \S 51, p. 274-307]{Top2}; see also \cite[Definition 9.3]{continua}.)
A distinction of points of a dendrite based on their order will be given later in \cref{sub:points:arcs}.

For the purpose of this work, we are going to consider the most regular non-trivial dendrites, which are called Wa\.zewski dendrites and are defined right below.
We would like to mention that a general dendrite might be so irregular that its homeomorphism group is trivial, so it is natural that the most regular dendrites would produce larger and often more interesting groups of homeomorphisms.

\begin{definition}
\label{def:wazewski:dendrite}
Given a natural number $n \geq 3$, a \textbf{Wa\.zewski dendrite} of order $n$, denoted by $D_n$, is a dendrite each of whose points is either of order $1$, $2$ or $n$ (no intermediate values), and such that every arc of $D_n$ (which is any subspace of $D_n$ that is homeomorphic to the closed interval $[0,1]$) contains a point of order $n$.
\end{definition}

For example, \cref{fig:D3} depicts the Wa\.zewski dendrite of order $3$.

By Theorem 6.2 of \cite{selfhomeomorphisms}, $D_n \simeq D_m$ if and only if $n = m$, so we can actually talk about \textit{the} Wa\.zewski dendrite of order $n$.
In the next Subsections we will build these dendrites as limit spaces of replacement systems.

\subsection{Replacement Systems}
\label{sub:rep:sys}

We will be working with finite graphs, all of which are directed edge-colored graphs where we allow both loops and parallel edges.
Thus, for our purpose a \textbf{graph} will be a sixtuplet $(V, E, C, \mathrm{i}, \mathrm{t}, \mathrm{c})$ where:
\begin{itemize}
    \item $V$, $E$ and $C$ are finite sets, whose elements are called \textbf{vertices}, \textbf{edges} and \textbf{colors}, respectively;
    \item $\mathrm{i}$ and $\mathrm{t}$ are maps $E \to V$, whose images $\mathrm{i}(e)$ and $\mathrm{t}(e)$ are called the \textbf{initial vertex} and the \textbf{terminal vertex} of the edge $e$;
    \item $\mathrm{c}$ is a map $E \to C$, and each image $\mathrm{c}(e)$ is called the \textbf{color} of the edge $e$.
\end{itemize}
When the set of colors is a singleton (which will be the case for dendrite replacement systems until \cref{sub:orientation}), the map $\mathrm{c}$ is trivial, and both the set $C$ and the map $\mathrm{c}$ can be naturally omitted.
In this case, we say that the replacement system is \textbf{monochromatic}.

We say that an edge $e$ is \textbf{incident} on a vertex $v$ when $\iota(e)=v$ or $\tau(e)=v$.
When two edges are incident on a common vertex, we say that they are \textbf{adjacent}.

\begin{definition}
\label{def:rep:sys}
A \textbf{replacement system} consists of the following data:
\begin{itemize}
    \item a finite set $C$ of colors, which will be coloring the edges of each graph mentioned below;
    \item a finite graph $\Gamma$, called the \textbf{base graph};
    \item for each color $c \in C$, a finite graph $R_c$, called the \textbf{$\boldsymbol{c}$ replacement graph} (for example, the \textit{red replacement graph}), equipped with distinct vertices $\iota$ and $\tau$ called the \textbf{initial} and \textbf{terminal vertices} of $R_c$, respectively.
\end{itemize}
\end{definition}

Starting from the base graph $\Gamma$, we can \textbf{expand} it by replacing one of its edges $e$ with the $c(e)$ replacement graph, attaching the initial and terminal vertices of $R_{c(e)}$ in place of the initial and terminal vertices $i(e)$ and $t(e)$ of the edge $e$.
The same procedure can be applied to any expansion of the base graph, which allows to obtain infinitely many further expansions of the base graph.
For example, \cref{fig:rep:sys:A} shows the Airplane replacement system (which realizes the Airplane rearrangement group $T_A$ from \cite[Example 2.13]{belk2016rearrangement}, studied in \cite{Airplane}), and \cref{fig:expansion:A} portrays two subsequent expansions of the base graph of this replacement system:
the first one is a blue expansion (and it is the only possible one, because the base graph consists of a sole edge) and the second one is a red expansion.
The order of expansions does not matter when expanding two edges of the same graph (for example, if we want to expand the edges $ib_1$ and $ib_3$ in one of the two graphs depicted in \cref{fig:expansion:A}, it does not matter which of the two we expand first.

\begin{figure}
\centering
\begin{tikzpicture}
    \node at (0,1.6) {$\Gamma$};
    \node[vertex] (l) at (-.75,0) {};
    \node[vertex] (r) at (.75,0) {};
    \draw[edge,blue] (l) to node[above]{$i$} (r);
    \begin{scope}[xshift=3.5cm]
    \node at (0,1.6) {$R_{\text{\textcolor{blue}{blue}}}$};
    \node[vertex] (l) at (-1.75,0) {}; \draw (-1.75,0) node[above]{$\iota$};
    \node[vertex] (cl) at (-.5,0) {};
    \node[vertex] (cr) at (.5,0) {};
    \node[vertex] (r) at (1.75,0) {}; \draw (1.75,0) node[above]{$\tau$};
    \draw[edge,blue] (cl) to node[above]{$b_1$} (l);
    \draw[edge,blue] (cr) to node[above]{$b_4$} (r);
    \draw[edge,red] (cr) to[out=90,in=90,looseness=1.4] node[above]{$b_2$} (cl);
    \draw[edge,red] (cl) to[out=270,in=270,looseness=1.4] node[above]{$b_3$} (cr);
    \end{scope}
    \begin{scope}[xshift=7.5cm]
    \node at (0,1.6) {$R_{\text{\textcolor{red}{red}}}$};
    \node[vertex] (l) at (-1.25,0) {}; \draw (-1.25,0) node[above]{$\iota$};
    \node[vertex] (r) at (1.25,0) {}; \draw (1.25,0) node[above]{$\tau$};
    \node[vertex] (c) at (0,0) {};
    \node[vertex] (ct) at (0,1) {};
    \draw[edge,red] (l) to node[below]{$r_1$} (c);
    \draw[edge,red] (c) to node[below]{$r_2$} (r);
    \draw[edge,blue] (c) to node[left]{$r_3$} (ct);
    \end{scope}
\end{tikzpicture}
\caption{The Airplane replacement system.}
\label{fig:rep:sys:A}
\end{figure}
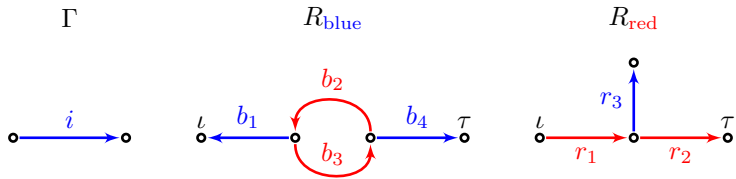

\begin{figure}
\centering
\begin{tikzpicture}
    \node[vertex] (l) at (-2.333,0) {};
    \node[vertex] (cl) at (-.667,0) {};
    \node[vertex] (cr) at (.667,0) {};
    \node[vertex] (r) at (2.333,0) {};
    \draw[edge,blue] (cl) to node[above]{$i b_1$} (l);
    \draw[edge,blue] (cr) to node[above]{$i b_4$} (r);
    \draw[edge,red] (cr) to[out=90,in=90,looseness=1.4] node[above]{$i b_2$} (cl);
    \draw[edge,red] (cl) to[out=270,in=270,looseness=1.4] node[above]{$i b_3$} (cr);
    \begin{scope}[xshift=6cm]
    \node[vertex] (l) at (-2.333,0) {};
    \node[vertex] (cl) at (-.667,0) {};
    \node[vertex] (c) at (0,.667) {};
    \node[vertex] (ct) at (0,1.667) {};
    \node[vertex] (cr) at (.667,0) {};
    \node[vertex] (r) at (2.333,0) {};
    \draw[edge,blue] (cl) to node[above]{$i b_1$} (l);
    \draw[edge,blue] (cr) to node[above]{$i b_4$} (r);
    \draw[edge,red] (cr) to[out=90,in=0] node[above right]{$i b_2 r_1$} (c);
    \draw[edge,blue] (c) to node[above left]{$i b_2 r_3$} (ct);
    \draw[edge,red] (c) to[out=180,in=90] node[above left]{$i b_2 r_2$} (cl);
    \draw[edge,red] (cl) to[out=270,in=270,looseness=1.45] node[above]{$i b_3$} (cr);
    \end{scope}
\end{tikzpicture}
\caption{Two subsequent expansions of the base graph of the Airplane replacement system.}
\label{fig:expansion:A}
\end{figure}
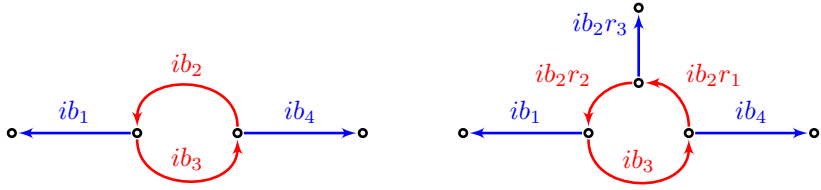

If we name each edge of the base and replacement graphs by distinct symbols in an alphabet $\mathbb{A}$, then we can represent the edges of an expansion of the base graph as certain finite words in the alphabet $\mathbb{A}$.
The aforementioned \cref{fig:expansion:A} shows, for example, that the edges newly obtained by performing the expansions are words of length two or three (in this case the first letter is always $i$ because $i$ is the only edge of the base graph, so it's the only possible first expansion).

\subsection{Limit Spaces}
\label{sub:limit:spaces}

The limit space of a replacement system is built as a quotient of a Cantor space $\Omega$ by an equivalence relation.
Here are their definitions.

\begin{definition}
Given a replacement system $\mathcal{R}$ and the alphabet $\mathbb{A}$ consisting of the edges of the base and replacement graphs of $\mathcal{R}$, the \textbf{symbol space} $\Omega$ of $\mathcal{R}$ is the set of all infinite words in $\mathbb{A}$ such that each finite prefix represents an edge of some expansion.
Moreover, we say that such prefixes are \textbf{allowed words}.
\end{definition}

In essence, an element of $\Omega$ is an infinite sequence of edges, each ``nested'' below the previous one via expansions.
This space is actually a one-sided shift of finite type, but there is no need to be familiar with that machinery to have a clear understanding of what $\Omega$ is.

We will only be considering replacement systems that satisfy the natural requirements expressed in the next Definition.

\begin{definition}
\label{def:expanding}
A replacement system $\mathcal{R}$ is \textbf{expanding} if
\begin{enumerate}
    \item its base and replacement graphs do not have isolated vertices;
    \item the initial and terminal vertices $\iota$ and $\tau$ of each replacement graph are not adjacent;
    \item each replacement graph has at least three vertices and two edges.
\end{enumerate}
\end{definition}

When equipped with the subspace topology inherited from the product topology of $\mathbb{A}^\infty$ (which is a Cantor space), the symbol space $\Omega$ is always compact, metrizable and totally disconnected.
Indeed, it inherits a basis of clopen sets from the cones of $\mathbb{A}^\infty$ (a \textbf{cone} is the subspace consisting of all of those sequences that have a fixed common prefix).
Then $\Omega$ is compact because it is an intersection of cones and it is metrizable and totally disconnected because $\mathbb{A}^\infty$ is.
Additionally, when the replacement system is expanding, $\Omega$ is a Cantor space, since it has no isolated point because every cone contains multiple elements of $\Omega$ by condition (3) of \cref{def:expanding}.
%

\begin{definition}
\label{def:gluing}
Given a replacement system $\mathcal{R}$ and its symbol space $\Omega$, we define the \textbf{gluing relation $\boldsymbol{\sim}$} as
\[ \alpha \sim \beta \iff \alpha_1 \dots \alpha_n \text{ and } \beta_1 \dots \beta_n \text{ are equal or adjacent edges in } \Gamma_n, \forall n \in \mathbb{N}, \]
where $\alpha = \alpha_1 \alpha_2 \dots$ and $\beta = \beta_1 \beta_2 \dots$ are sequences in $\Omega$, and $\Gamma_n$ is the sequence of graphs obtained by expanding, at each step, every edge of the previous graph, starting from the base graph $\Gamma$ (thus, $\Gamma_n$ contains precisely those edges that correspond to allowed finite words of length $n$).
\end{definition}

By \cite[Proposition 1.9]{belk2016rearrangement}, the gluing relation defined by an expanding replacement system is an equivalence relation, so the following definition makes sense.

\begin{definition}
\label{def:limit:space}
Given an expanding replacement system $\mathcal{R}$, its \textbf{limit space} is the quotient topological space of its symbol space $\Omega$ by the gluing relation $\sim$.
For all $\omega \in \Omega$, we will denote its image via the quotient map by $\llbracket \omega \rrbracket$.
\end{definition}

Continuing with our example, the limit space of the Airplane replacement system, simply called the Airplane limit space, is depicted in \cref{fig:airplane:limit:space}.
It is homeomorphic to the Julia set for the complex map $z \mapsto z^2 - 1.755$.
In \cref{sub:den} we will build replacement systems whose limit spaces are Wa\.zewski dendrites.

\begin{figure}
\centering
\includegraphics{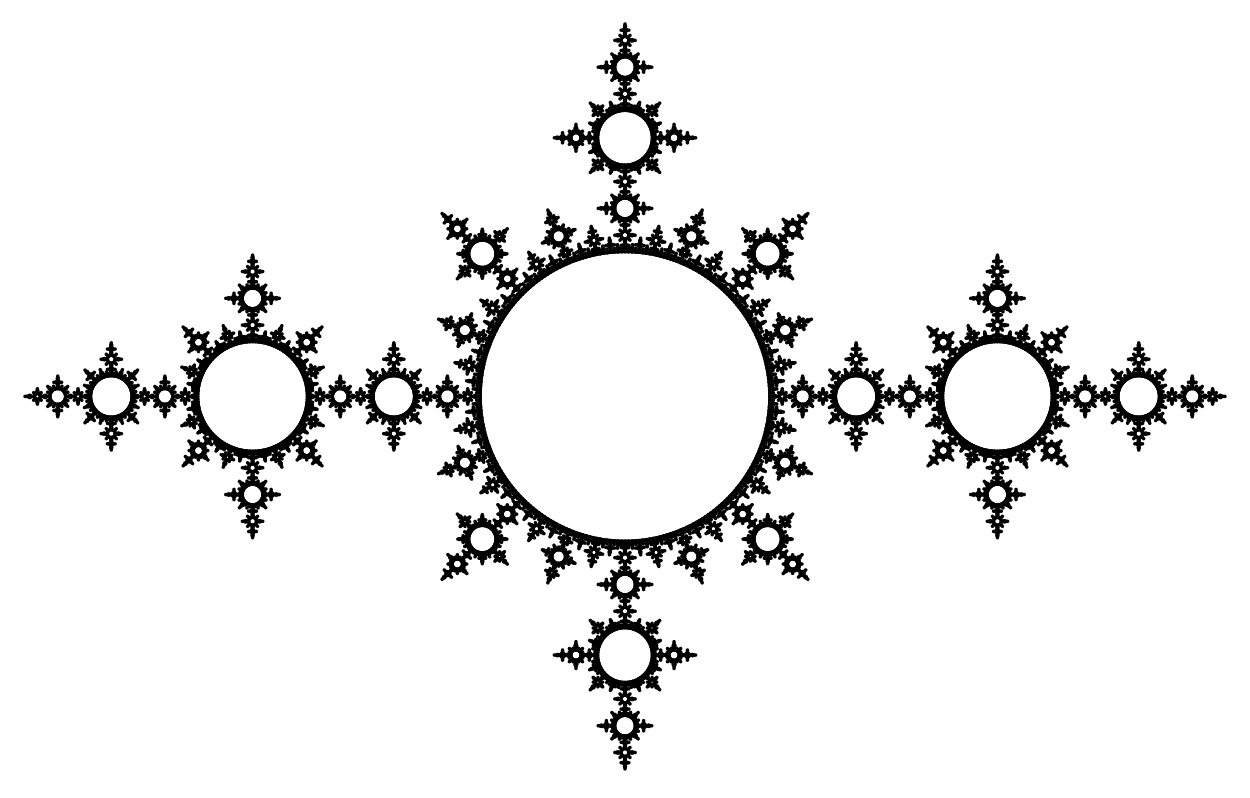}
\caption{The Airplane limit space.}
\label{fig:airplane:limit:space}
\end{figure}

By \cite[Theorem 1.25]{belk2016rearrangement}, limit spaces are compact metric spaces.
We remark here that, while fractals are usually thought of subsets of a Euclidean space, the limit spaces produced here are only defined as topological spaces.
For example, under this point of view the famous Koch snowflake is to be thought of as topologically equivalent to a circle.

\subsection{Rearrangement Groups}
\label{sub:rearrangements}

Given an expanding replacement system $\mathcal{R}$ and its limit space $X$, consider an edge $e$ of some expansion of the base graph (i.e., any allowed finite word).
We define the \textbf{cell $\boldsymbol{C(e)}$} as the image of the cone of those sequences of the symbol space that have $e$ as a prefix via the quotient map of the gluing relation.
It will often be useful to consider the topological interior of a cell $C(e)$, which will be denoted by $C^\mathrm{o}(e)$.

Each cell $C(e)$ has a \textbf{type}, distinguished by whether the edge $e$ is a loop or not and on its color.
For example, the Airplane limit space admits two types or cells, only distinguished by their color, as there are no loops among the expansions of the base graph.
Two cells $C(e_1)$ and $C(e_2)$ of the same type admit a natural \textbf{canonical homeomorphism}, which is given by the map that changes the prefix $e_1$ with $e_2$.
Depending on whether it is a loop or not, a cell $C(e)$ has one or two endpoints, which correspond to the initial and terminal vertices of the edge $e$.

To each expansion $\Lambda$ of the base graph we associate a \textbf{cellular partition} of the limit space $X$ (i.e., a cover of $X$ by finitely many cells whose interiors are disjoint, meaning that they may only intersect at the endpoints of the cells) as follows:
\[ \Lambda \mapsto \{ C(e) \mid e \text{ is an edge of } \Lambda \}, \]
which is a bijection between the set of expansions of the base graph and that of cellular partitions.
The notions of canonical homeomorphisms and cellular partitions prompt the following definition:

\begin{definition}
\label{def:rearrangement}
Given an expanding replacement system $\mathcal{R}$ with limit space $X$, a \textbf{rearrangement} is a homeomorphism of $X$ that restricts to a canonical homeomorphism between the cells of some cellular partition.
\end{definition}

It is not hard to see that rearrangements form a group under composition, so we can talk about \textbf{rearrangement groups} of a given expanding replacement system.

Rearrangements naturally correspond to graph isomorphisms (which need to preserve edge orientation and colors) between two expansions of the base graph:
if the graph isomorphism maps an edge $e_1$ to $e_2$, then the rearrangement restricts to a canonical homeomorphism from the cell $C(e_1)$ to $C(e_2)$.
A triple $(D,R,f)$ given by two expansions $D$ and $R$ of the base graph and a graph isomorphism $f: D \to R$ is called a \textbf{graph pair diagram} for the induced rearrangement.
Each rearrangement has multiple graph pair diagrams that represent it (one can always ``expand'' a graph pair diagram by expanding an edge $e$ in the domain graph and its image under the graph isomorphism in the range graph), but there always exists a unique reduced graph pair diagram for each rearrangement.

Finally, we remark here that rearrangement groups act properly by isometries on CAT(0) cubical complexes, as explained in Section 3 of \cite{belk2016rearrangement}.
In particular, an application of \cite[Theorem 1.2]{MR2671183} to the action of a rearrangement group on the 1-skeleton of this complex yields that rearrangement groups do not have Kazhdan's Property (T), nor do any of their infinite subgroups (such as the commutator subgroup of dendrite rearrangement groups).

\subsection{Thompson Groups}

Thompson's group $F$ is the group of those orientation-preserving homeomorphisms of the unit interval $[0,1]$ that are piecewise linear with slopes that are powers of $2$ and breakpoints that are dyadic (i.e., they can be written as fractions whose denominator is a power of $2$, or equivalently they belong to $\mathbb{Z}[\frac{1}{2}]$).
Thompson groups $T$ and $V$ can be defined in a similar fashion, with the difference that $T$ acts on the unit circle and $V$ on the Cantor space in place of the unit interval.
We refer to \cite{cfp} for more general information about Thompson groups, while \cite{ThF} is specifically about $F$.

As mentioned earlier, Thompson groups $F$, $T$ and $V$ can all be realized as rearrangement groups;
\cref{fig:thompson:F} depicts a replacement system for $F$, which in the following is going to be featured much more prominently than the other two.
In particular, we are going to frequently use the transitivity properties of $F$ on the set of dyadic points of $[0,1]$ and the fact that $F$ is generated by the two elements portrayed in \cref{fig:thompson:F:gen}.
Moreover, in \cref{sec:comm} we are going to use the fact that the commutator subgroup $[F,F]$ consists precisely of those elements of $F$ that act trivially in some neighbourhood of $0$ and of $1$.

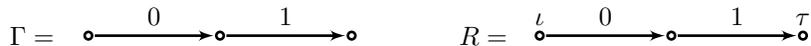
\begin{figure}
\centering
\begin{tikzpicture}
    \node at (-.75,0) {$\Gamma =$};
    \node[vertex] (L) at (0,0) {};
    \node[vertex] (C) at (1.75,0) {};
    \node[vertex] (R) at (3.5,0) {};
    \draw[edge] (L) to node[above]{$0$} (C);
    \draw[edge] (C) to node[above]{$1$} (R);
    \begin{scope}[xshift=6cm]
    \node at (-.75,0) {$R =$};
    \node[vertex] (L) at (0,0) {}; \draw (0,0) node[above]{$\iota$};
    \node[vertex] (C) at (1.75,0) {};
    \node[vertex] (R) at (3.5,0) {}; \draw (3.5,0) node[above]{$\tau$};
    \draw[edge] (L) to node[above]{$0$} (C);
    \draw[edge] (C) to node[above]{$1$} (R);
    \end{scope}
\end{tikzpicture}
\caption{A replacement system for Thompson's group $F$.}
\label{fig:thompson:F}
\end{figure}

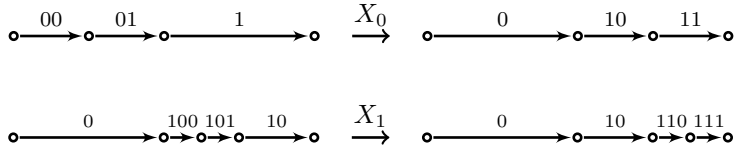
\begin{figure}
\centering
\begin{tikzpicture}
    \node[vertex] (0) at (0,0) {};
    \node[vertex] (1/4) at (1,0) {};
    \node[vertex] (1/2) at (2,0) {};
    \node[vertex] (1) at (4,0) {};
    \draw[edge] (0) to node[above]{\footnotesize$00$} (1/4);
    \draw[edge] (1/4) to node[above]{\footnotesize$01$} (1/2);
    \draw[edge] (1/2) to node[above]{\footnotesize$1$} (1);
    \draw[->] (4.5,0) to node[above]{$X_0$} (5,0);
    \begin{scope}[xshift=5.5cm]
    \node[vertex] (0) at (0,0) {};
    \node[vertex] (1/2) at (2,0) {};
    \node[vertex] (3/4) at (3,0) {};
    \node[vertex] (1) at (4,0) {};
    \draw[edge] (0) to node[above]{\footnotesize$0$} (1/2);
    \draw[edge] (1/2) to node[above]{\footnotesize$10$} (3/4);
    \draw[edge] (3/4) to node[above]{\footnotesize$11$} (1);
    \end{scope}
\end{tikzpicture}
\\
\vspace{20pt}
\begin{tikzpicture}
    \node[vertex] (0) at (0,0) {};
    \node[vertex] (1/2) at (2,0) {};
    \node[vertex] (5/8) at (2.5,0) {};
    \node[vertex] (3/4) at (3,0) {};
    \node[vertex] (1) at (4,0) {};
    \draw[edge] (0) to node[above]{\scriptsize$0$} (1/2);
    \draw[edge] (1/2) to node[above]{\scriptsize$100$} (5/8);
    \draw[edge] (5/8) to node[above]{\scriptsize$101$} (3/4);
    \draw[edge] (3/4) to node[above]{\scriptsize$10$} (1);
    \draw[->] (4.5,0) to node[above]{$X_1$} (5,0);
    \begin{scope}[xshift=5.5cm]
    \node[vertex] (0) at (0,0) {};
    \node[vertex] (1/2) at (2,0) {};
    \node[vertex] (3/4) at (3,0) {};
    \node[vertex] (7/8) at (3.5,0) {};
    \node[vertex] (1) at (4,0) {};
    \draw[edge] (0) to node[above]{\scriptsize$0$} (1/2);
    \draw[edge] (1/2) to node[above]{\scriptsize$10$} (3/4);
    \draw[edge] (3/4) to node[above]{\scriptsize$110$}  (7/8);
    \draw[edge] (7/8) to node[above]{\scriptsize$111$} (1);
    \end{scope}
\end{tikzpicture}
\caption{The generators $X_0$ and $X_1$ of Thompson's group $F$.}
\label{fig:thompson:F:gen}
\end{figure}


\section{Rearrangement Groups for Wa\.zewski Dendrites}
\label{sec:dendrite:rearrangement:groups}

In this Section we introduce the main topic of this work:
the Thompson-like groups $G_n$, which will be defined in \cref{sub:den} as rearrangement groups starting from their replacement systems $\mathcal{D}_n$ for the Wa\.zewski dendrites $D_n$.
In \cref{sub:points:arcs} we define branch points, endpoints and (rational) arcs of $D_n$, which will be important in the remainder of this work.
Finally, in \cref{sub:subgroups} we start discussing two fundamental subgroups of the $G_n$'s:
a copy of $\mathrm{Sym}(n)$ that rigidly permutes the branches around a specific branch point $p_0$ and a copy of Thompson's group $F$ acting on a certain linear portion of the dendrite.

\subsection{Replacement Systems for Dendrites}
\label{sub:den}

In this Subsection we define the replacement systems $\mathcal{D}_n$ that produce the Wa\.zewski dendrites $D_n$ as limit spaces, for any natural number $n \geq 3$.
\cref{fig:rep:sys} schematically depicts the generic replacement system $\mathcal{D}_n$, which is defined with more precision as follows.
The \textbf{dendrite replacement system} $\boldsymbol{\mathcal{D}_n}$ is the monochromatic replacement system whose base and replacement graphs are the same tree consisting of a vertex of degree $n$ that is the origin of $n$ edges terminating at $n$ distinguished leaves;
the initial and terminal vertices $\iota$ and $\tau$ of the replacement graph are two distinct leaves.

\begin{figure}
\centering
\begin{tikzpicture}
    \node at (-2,0) {$\Gamma =$};
    \node[vertex] (c) at (0,0) {};
    \node[vertex] (1) at (0:1.5) {};
    \node[vertex] (2) at (72:1.5) {};
    \node[vertex] (3) at (144:1.5) {};
    \node[vertex] (n-1) at (216:1.5) {};
    \node[vertex] (n) at (288:1.5) {};
    \draw[edge] (c) to node[above]{$1$} (1);
    \draw[edge] (c) to node[above left]{$2$} (2);
    \draw[edge,dotted] (c) to (3);
    \draw[edge,dotted] (c) to (n-1);
    \draw[edge] (c) to node[right]{$n$} (n);
    \begin{scope}[xshift=7cm]
    \node at (-2.5,0) {$R =$};
    \node[vertex] (b) at (0,0) {};
    \node[vertex] (i) at (180:1.5) {}; \draw (0:-1.5) node[above]{$\iota$};
    \node[vertex] (r2) at (120:1.5) {};
    \node[vertex] (rn-1) at (60:1.5) {};
    \node[vertex] (t) at (0:1.5) {}; \draw (0:1.5) node[above]{$\tau$};
    \draw[edge] (b) to node[below]{$1$} (i);
    \draw[edge,dotted] (b) to (r2);
    \draw[edge,dotted] (b) to (rn-1);
    \draw[edge] (b) to node[below]{$n$} (t);
    \end{scope}
\end{tikzpicture}
\caption{A schematic depiction of the dendrite replacement system $\mathcal{D}_n$ (the dotted lines are meant to indicate that each graph has a total of $n$ edges).}
\label{fig:rep:sys}
\end{figure}
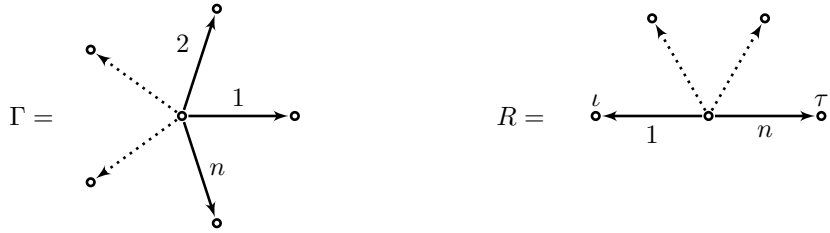

\phantomsection\label{txt:rep:sys:den}
The replacement system $\mathcal{D}_n$ is expanding (\cref{def:expanding}), so the gluing relation is an equivalence relation (Propositions 1.9 and 1.22 of \cite{belk2016rearrangement}) and the limit space of $\mathcal{D}_n$ is defined as in \cref{def:limit:space}.
When endowed with the quotient topology induced by the gluing relation (see Subsection 1.5 of \cite{belk2016rearrangement}), the limit space is a non-empty compact metrizable space (by \cite[Theorem 1.25]{belk2016rearrangement}), it is connected (because both the base graph and the replacement graphs are connected, see \cite[Remark 1.27]{belk2016rearrangement}) and it is locally connected (proven in \cref{lem:loc:conn} below).
It contains no simple closed curve, as none of the expansions of $\mathcal{D}_n$ contain simple closed paths.
Thus, the limit space of $\mathcal{D}_n$ is a dendrite (\cref{def:dendrite}), and since its points have order $1$, $2$ or $n$, it is homeomorphic to the Wa\.zewski dendrite $D_n$ (\cref{def:wazewski:dendrite}).
We will thus use the same symbol $D_n$ to denote both the Wa\.zewski dendrite and the limit space of $\mathcal{D}_n$.

\begin{remark}
\label{rmk:undirected}
Note that the edges of $\mathcal{D}_n$ are \textbf{undirected}, by which we mean that we can completely forget about their edge orientation without losing any meaningful information.
This is because the expansion of an edge going from a vertex $v$ to a vertex $w$ would be exactly the same if that edge were instead going from $w$ to $v$.
More detail and examples of this phenomenon can be found at the end of Subsections 3.2 of \cite{Airplane} and in \cite[Remark 1.5]{RearrConj}.
In practice, this simply means that we will be able to forget about the orientation of every edge altogether when representing rearrangements, because if an edge $x$ is mapped to an edge $y$ and the orientation is swapped, then what is really happening is $xi \mapsto x \tilde{i}$, where $\tilde{i} = n+1-i$, for all $i \in \{1, \dots, n\}$.
The orientation is still important to correctly describe the gluing relation and codify cells.
\end{remark}

\subsection{Branch Points, Branches, Endpoints and Arcs}
\label{sub:points:arcs}

In this Subsection we give a few definition that will be useful throughout this work.

\begin{definition}
\label{def:branch}
A \textbf{branch point} of $D_n$ is a point $p \in D_n$ such that $D_n \setminus \{p\}$ has exactly $n$ connected components.
Each of these connected components is called a \textbf{branch} of $D_n$ at the branch point $p$.
\end{definition}

Observe that each branch at $p$ is the topological interior of the union of finitely many cells, one of which has boundary $\{ p \}$.
However, note that not every topological interior of the connected union of finitely many cells is a branch.

We will often refer to the following specific branch point:
\[ \llbracket 11\bar{n} \rrbracket =  \llbracket 21\bar{n} \rrbracket = \dots = \llbracket n1\bar{n} \rrbracket, \]
where the bar on top of $n$ denotes periodicity.
We call this point the \textbf{central branch point} and we denote by $\boldsymbol{p_0}$.
There is nothing topologically special about this branch point, but its nice combinatorial description will make it handy in the remainder of this work.

\begin{definition}
\label{def:ext}
An \textbf{endpoint} of $D_n$ is a point $q \in D_n$ such that $D_n \setminus \{q\}$ has exactly one connected component.
\end{definition}

Given a subset $X$ of $D_n$, we denote by $\boldsymbol{\mathrm{Br}(X)}$ and $\boldsymbol{\mathrm{En}(X)}$ the sets of all branch points and endpoints, respectively, that are contained in $X$.
When $X = D_n$ itself, we simply write $\boldsymbol{\mathrm{Br}}$ and $\boldsymbol{\mathrm{En}}$.

\begin{remark}
\label{rmk:br:en:dns}
\cite[Remark 3.2]{IG} states that a subset of the limit space is dense if and only if it intersects non-trivially every cell.
Since every cell contains branch points and endpoints, we can conclude that both $\mathrm{Br}(X)$ and $\mathrm{En}(X)$ are dense in $X$ whenever $X$ is a cell or the entire limit space $D_n$.
It is also useful to note that $\mathrm{Br}$ is countable, whereas $\mathrm{En}$ is uncountable.
\end{remark}

Note that every vertex of a graph expansion of $\mathcal{D}_n$ either corresponds to a branch point or an endpoint, depending on whether its degree is $n$ or $1$.
Moreover, every branch point is a vertex of some graph expansion, but not every endpoint is a vertex, since $\mathrm{En}$ is uncountable.
For concrete examples of endpoints that are not vertices, consider any sequence that does not contain $1$ and that is not eventually repeating.

We say that an endpoint is a \textbf{rational endpoint} if it is a vertex of some graph expansion.
We denote by $\boldsymbol{\mathrm{REn}(X)}$ the set of those rational endpoints that are contained in $X \subseteq D_n$, and we simply write $\boldsymbol{\mathrm{REn}}$ when $X$ is the entire dendrite $D_n$.
Note that, by the same reason mentioned in \cref{rmk:br:en:dns}, $\mathrm{REn}$ is dense in $D_n$ because every cell contains rational endpoints.

Being a rational endpoint is not a topological property:
the set $\mathrm{REn}$ depends on the chosen homeomorphism between the limit space and the Wa\.zewski dendrite $D_n$, and in general a homeomorphism need not preserve the rationality of an endpoint.
Since we have already fixed a homeomorphism between the limit space and the dendrite to begin with, there will be no need to distinguish between endpoints that are rational under a homeomorphism but not under another.

Differently from generic homeomorphisms, it is easy to see that an element of the rearrangement group $G_n$ must induce a permutation of $\mathrm{REn}$.

On the other hand, every branch point of $D_n$ is rational, since it is a vertex of some graph expansion.
In this sense, branch points are more useful when dealing with rearrangements, since the set of branch points is invariant under any homeomorphism.

\begin{remark}
\label{rmk:br:ex:pts}
Branch points and rational endpoints have nice and useful combinatorial characterizations based on the sequences that represent them in $\mathcal{D}_n$.
\begin{itemize}
    \item Each branch point of $D_n$ corresponds exactly to the $n$ sequences
    \[ x 1 1 \bar{n}, \: x 2 1 \bar{n}, \: \dots, \: x n 1 \bar{n} \]
    for a unique finite word $x$ (possibly empty).
    In particular, there are natural bijections between the set $\mathrm{Br}$ of branch points of $D_n$, the set of finite words $x$ in the alphabet $\{1, \dots, n\}$ (including the empty word) and the set that consists of the cells $C(x)$ of $\mathcal{D}_n$ together with the entire $D_n$ (which can be thought of as the cell associated to the empty word).
    \item Each rational endpoint of $D_n$ corresponds uniquely to a sequence
    \[ i \bar{n} \text{ or } x \bar{n} \]
    for a unique $i \in \{1, \dots, n\}$, or for a unique finite non-empty word $x$ in the alphabet $\{1, \dots, n\}$ that does not end with $1$ nor with $n$.
\end{itemize}
\end{remark}

\begin{lemma}
\label{lem:loc:conn}
For all $n \geq 3$, the limit space $D_n$ is locally connected.
\end{lemma}

\begin{proof}
Given a point $p \in D_n$, we want to build a basis of connected open sets at $p$.
We distinguish the case in which $p$ corresponds to a vertex and the case in which it does not, and we will use the fact that each cell is connected (because the replacement graph is connected).

If $p$ corresponds to a vertex of some graph expansion, then it is codified by a finite amount of sequences $x \alpha^{(1)}, \dots, x \alpha^{(m)}$ with a common prefix $x$ (possibly empty) and $m$ infinite sequences $\alpha^{(1)}, \dots, \alpha^{(m)}$.
More precisely, by \cref{rmk:br:ex:pts} if $p$ is a branch point then $m=n$ and if $p$ is a rational endpoint then $m=1$.
In both cases, a basis of connected open sets is given by the cells corresponding to the edges that are adjacent to $p$ in the $k$-th full expansion of $\mathcal{D}_n$ in the following way:
\[ \left\{ \{p\} \cup C^\mathrm{o}(x \alpha^{(1)}_1 \dots \alpha^{(1)}_k) \cup \dots \cup C^\mathrm{o}(x \alpha^{(m)}_1 \dots \alpha^{(m)}_k) \mid k \in \mathbb{N} \right\}. \]

Now consider a point $p$ that is not a vertex.
It corresponds to a unique infinite sequence $\omega = \omega_1 \omega_2 \dots$, and must belong to the interior of every $C(\omega_1 \dots \omega_k)$, so a basis of connected open sets is simply given by
\[ \{ C^\mathrm{o}(\omega_1 \dots \omega_k) \mid k \in \mathbb{N} \}. \]
\end{proof}

\begin{definition}
\label{def:arc}
Since dendrites do not contain simple closed curves, any two points are joined by a unique topological arc, which we simply call an \textbf{arc} of $D_n$.
Given two points $p$ and $q$, we denote by $\boldsymbol{[p,q]}$ the unique arc joining them.
\end{definition}

We already mentioned arcs in \cref{def:wazewski:dendrite}, describing them as homeomorphic copies of $[0,1]$;
naturally these two definitions are equivalent, but it is clear that, for dendrites, the existence and uniqueness of arcs for each pair of distinct points is worth remarking and plays a significant role in their topological properties and in the kind of homeomorphisms that are available on them.
Indeed, this notion will be important for identifying the Thompson subgroups described later in \cref{sub:thomp}, where we will consider the following special types of arcs.

\begin{definition}
\label{def:arc:type}
We say that an arc $A = [p,q]$ is a \textbf{rational arc} if $p$ and $q$ are vertices of some graph expansion of $\mathcal{D}_n$.
Each rational arc $[p,q]$ is of one of the following three types:
\begin{itemize}
    \item it is a \textbf{BB-arc} (\textit{branch point to branch point}) if $p, q \in \mathrm{Br}$;
    \item it is an \textbf{EE-arc} (\textit{rational endpoint to rational endpoint}) if $p, q \in \mathrm{REn}$;
    \item it is an \textbf{BE-arc} (\textit{branch point to rational endpoint}) if $p \in \mathrm{Br}$ and $q \in \mathrm{REn}$ or if $p \in \mathrm{REn}$ and $q \in \mathrm{Br}$.
\end{itemize}
\end{definition}

Similarly to the rationality of an endpoint, the rationality of an arc is not a topological property, meaning that homeomorphisms need not preserve it.
It is true, however, that a rearrangement must map a rational arc to a rational arc and a non-rational arc to a non-rational arc.
In practice, we will only be working with arcs that are rational throughout the rest of this paper.

We will frequently use the following correspondence between paths in a graph expansion and rational arcs.

\begin{remark}
\label{rmk:paths}
There is a natural surjective correspondence from the set of paths in graph expansions to the set of all rational arcs, which is the following.
If $e = x_1 \dots x_k$ is an edge of some graph expansion, we associate to it the arc $[p,q]$ for $p = \llbracket x_1 \dots x_k 1 \bar{n} \rrbracket$ and $q = \llbracket x_1 \dots x_k \bar{n} \rrbracket$.
This arc is precisely the set of points of $D_n$ that are represented by some sequence $x_1 \dots x_k \omega$ for $\omega$ in the alphabet $\{1,n\}$.
Inductively extending this construction gives the desired correspondence from paths of a graph expansions to rational arcs, which is surjective because clearly we can find a path between any two points of $D_n$ that are vertices.
In short, if $P$ is the path consisting of the edges $e_1, \dots, e_l$ (all of which are finite words), then the associated arc is
\[ A = \big\{ \llbracket e_i \omega \rrbracket \mid i=1, \dots, n \text{ and } \omega \text{ is a sequence in } \{1,n\} \big\}. \]
This correspondence is not injective, but the preimage of each rational arc contains a unique path with a minimal number of edges, which is obtained from any other edge of the same preimage by performing as many graph reductions as possible.
When restricted to the set of these minimal paths, this correspondence is a bijection.
\end{remark}

\begin{remark}
\label{rmk:arcs:dyadic}
For each EE-arc $A$, there is a canonical homeomorphism between $A$ and the unit interval $[0,1]$ that induces a bijection between the branch points of $A$ and the dyadic points of $(0,1)$.
The homeomorphism is built as described below using the path $P$ from the previous \cref{rmk:paths}.

Given a single edge $e = x_1 \dots x_l$, we consider the homeomorphism
\[ \Phi_e: \big\{ \llbracket e \omega \rrbracket \mid \omega \text{ is a sequence in } \{1,n\} \big\} \to \big[ 0, 2^{-l} \big] \]
that is built inductively by letting $e 1$ and $e n$ correspond to $2^{-l} [0,1/2]$ and $2^{-l} [1/2,1]$ in the natural order-preserving fashion, and so forth for further edge expansions.

Given a path $P$ of edges $e_1, \dots, e_k$, we build $\Phi_P$ as the concatenation of the $\Phi_{e_i}$'s, by which we mean that the first $\Phi_{e_1}$ maps to $[0, 2^{-l_1}]$, the second $\Phi_{e_2}$ to the adjacent interval of length $2^{-l_2}$ (which is $[2^{-l_1}, 2^{-l_1}+2^{-l_2}]$) and so forth.
For a generic path $P$, this covers some dyadic interval $[0, z] \subseteq [0,1]$, for some $z \in \mathbb{Z}[1/2]$, and it is easy to see that $\Phi_P (A) = [0,1]$ if and only if $P$ corresponds to an EE-arc.

One may note that this construction is very similar to the usual homeomorphism between $[0,1]$ and the limit space of the replacement system for Thompson's group $F$ (depicted in \cref{fig:thompson:F}), which is simply given by identifying each number in $[0,1]$ with its binary expansion.
For the replacement system of dendrite replacement systems (\cref{fig:rep:sys}), this is less straightforward due to the fact that the edge $1$ is directed towards the initial vertex.
\end{remark}

\subsection{Fundamental Subgroups of Dendrite Rearrangement Groups}
\label{sub:subgroups}

From here onward, given an $n \geq 3$, we will denote by $G_n$ the rearrangement group of $\mathcal{D}_n$, and we now begin its study.

In general, even if a replacement system produces an appealing limit space, its rearrangement group might be uninteresting or even trivial (see \cite[Example 2.5]{belk2016rearrangement}).
Here, however, we prove that every dendrite replacement system yields a rearrangement group that contains many copies of Thompson's group $F$ (\cref{sub:thomp}), immediately hinting at the fact that these rearrangement groups are, in fact, quite large.
How large they are will be made clear with \cref{thm:dense}, stating that each dendrite rearrangement group $G_n$ is dense in the full group of homeomorphisms of the Wa\.zewski dendrite $D_n$.

\subsubsection{The Permutation Subgroup \texorpdfstring{$K$}{K}}
\label{sub:perm}

The first fundamental subgroup of $G_n$ that we discuss is an isomorphic copy of the symmetric group that acts on the branches at the central branch point $p_0 = \llbracket i 1 \bar{n} \rrbracket$ (for any choice of $i \in \{1, \dots, n\}$).

\begin{lemma}
\label{lem:K}
For every $n \geq 3$, the rearrangement group $G_n$ of $\mathcal{D}_n$ contains an isomorphic copy $K$ of the symmetric group $\mathrm{Sym}(n)$ that acts on $D_n$ by canonically permuting the branches at $p_0$.
\end{lemma}

\begin{proof}
The set of those rearrangements whose graph pair diagrams have domain and range graphs equal to the base graph is a subgroup of $G_n$.
It is isomorphic to the automorphism group of the base graph, which is a copy of $\mathrm{Sym}(n)$ that canonically permutes the set $\{C(1),\dots,C(n)\}$.
\end{proof}

Rearrangements can freely permute the branches at any branch point, not only at $p_0$, as the following Lemma shows.

\begin{lemma}
\label{lem:perm}
For all $p \in \mathrm{Br}$, the stabilizer of $p$ in $G_n$ acts on the set of branches at $p$ as $\mathrm{Sym}(n)$.
\end{lemma}

\begin{proof}
Consider a branch point $p$ and let $E$ be the minimal graph expansion of $\mathcal{D}_n$ where $p$ appears.
Let $\Lambda$ be the subgraph of $E$ corresponding to the connected component of $E \setminus \{p\}$ that contains the central branch point $p_0$.
The subgraphs corresponding to the other $n-1$ connected components of $E \setminus \{p\}$ can be expanded in such a way that they become isomorphic to $\Lambda$.
We denote by $E^*$ the graph obtained by performing such expansions, where the subgraphs $\Lambda_i$ corresponding to the connected components of $E^* \setminus \{p\}$ are all isomorphic as trees rooted at $p$.
Now, clearly there are graph pair diagrams with domain and range graphs equal to $E^*$ that fix $p$ and realize any permutation of $\{\Lambda_1,\dots,\Lambda_n\}$.
\end{proof}

\subsubsection{The Thompson Subgroup \texorpdfstring{$H$}{H}}
\label{sub:thomp}

The second fundamental subgroup of $G_n$ is an isomorphic copy of Thompson's group $F$.
Let $A_0$ denote the EE-arc $[q_1, q_n]$ between the endpoints $q_1 = \llbracket 1 \bar{n} \rrbracket$ and $q_n = \llbracket \bar{n} \rrbracket$.
Note that, by \cref{rmk:paths}, $A_0$ corresponds to the set of those points of $D_n$ that are represented by some infinite sequence in the alphabet $\{ 1,n \}$.
The canonical homeomorphism from $A_0$ to $[0,1]$ described in \cref{rmk:arcs:dyadic} prompts the following Lemma.

\begin{lemma}
\label{lem:thomp}
For every $n \geq 3$, the dendrite rearrangement group $G_n$ contains an isomorphic copy $H$ of Thompson's group $F$ that acts on $A_0$ as $F$ does on the unit interval.
In particular, this subgroup $H$ is generated by the two elements portrayed in \cref{fig:H:generators}.
\end{lemma}

\begin{proof}
Consider the elements $g_0$ and $g_1$ depicted in \cref{fig:H:generators} (a precise definition in terms of sequences is given below) and let $H = \langle g_0, g_1 \rangle$.
A quick comparison with the standard generators $X_0$ and $X_1$ of Thompson's group $F$ (\cref{fig:thompson:F:gen}) immediately shows that the subgroup $H$ is isomorphic to Thompson's group $F$.
\end{proof}

\begin{figure}
\centering
\begin{tikzpicture}
    \node[vertex] (0) at (0,0) {};
    \node[vertex] (1/4) at (1,0) {};
        \node[vertex] (1/4 1) at ($(1/4)+(120:.5)$) {};
        \node[vertex] (1/4 2) at ($(1/4)+(60:.5)$) {};
    \node[vertex] (1/2) at (2,0) {};
        \node[vertex] (1/2 1) at ($(1/2)+(120:1)$) {};
        \node[vertex] (1/2 2) at ($(1/2)+(60:1)$) {};
    \node[vertex] (1) at (4,0) {};
    \draw (0) to (1/4);
    \draw (1/4) to (1/2);
        \draw[dotted] (1/4) to (1/4 1);
        \draw[dotted] (1/4) to (1/4 2);
    \draw (1/2) to (1);
        \draw[dotted] (1/2) to (1/2 1);
        \draw[dotted] (1/2) to (1/2 2);
    \draw[->] (4.5,0) to node[above]{$g_0$} (5,0);
    \begin{scope}[xshift=5.5cm]
    \node[vertex] (0) at (0,0) {};
    \node[vertex] (1/2) at (2,0) {};
        \node[vertex] (1/2 1) at ($(1/2)+(120:1)$) {};
        \node[vertex] (1/2 2) at ($(1/2)+(60:1)$) {};
    \node[vertex] (3/4) at (3,0) {};
        \node[vertex] (3/4 1) at ($(3/4)+(120:.5)$) {};
        \node[vertex] (3/4 2) at ($(3/4)+(60:.5)$) {};
    \node[vertex] (1) at (4,0) {};
    \draw (0) to (1/2);
    \draw (1/2) to (3/4);
        \draw[dotted] (1/2) to (1/2 1);
        \draw[dotted] (1/2) to (1/2 2);
    \draw (3/4) to (1);
        \draw[dotted] (3/4) to (3/4 1);
        \draw[dotted] (3/4) to (3/4 2);
    \end{scope}
\end{tikzpicture}
\\
\vspace{15pt}
\begin{tikzpicture}
    \node[vertex] (0) at (0,0) {};
    \node[vertex] (1/2) at (2,0) {};
        \node[vertex] (1/2 1) at ($(1/2)+(120:1)$) {};
        \node[vertex] (1/2 2) at ($(1/2)+(60:1)$) {};
    \node[vertex] (5/8) at (2.5,0) {};
        \node[vertex] (5/8 1) at ($(5/8)+(120:.25)$) {};
        \node[vertex] (5/8 2) at ($(5/8)+(60:.25)$) {};
    \node[vertex] (3/4) at (3,0) {};
        \node[vertex] (3/4 1) at ($(3/4)+(120:.5)$) {};
        \node[vertex] (3/4 2) at ($(3/4)+(60:.5)$) {};
    \node[vertex] (1) at (4,0) {};
    \draw (0) to (1/2);
    \draw (1/2) to (5/8);
        \draw[dotted] (1/2) to (1/2 1);
        \draw[dotted] (1/2) to (1/2 2);
    \draw (5/8) to (3/4);
        \draw[dotted] (5/8) to (5/8 1);
        \draw[dotted] (5/8) to (5/8 2);
    \draw (3/4) to (1);
        \draw[dotted] (3/4) to (3/4 1);
        \draw[dotted] (3/4) to (3/4 2);
    \draw[->] (4.5,0) to node[above]{$g_1$} (5,0);
    \begin{scope}[xshift=5.5cm]
    \node[vertex] (0) at (0,0) {};
    \node[vertex] (1/2) at (2,0) {};
        \node[vertex] (1/2 1) at ($(1/2)+(120:1)$) {};
        \node[vertex] (1/2 2) at ($(1/2)+(60:1)$) {};
    \node[vertex] (3/4) at (3,0) {};
        \node[vertex] (3/4 1) at ($(3/4)+(120:.5)$) {};
        \node[vertex] (3/4 2) at ($(3/4)+(60:.5)$) {};
    \node[vertex] (7/8) at (3.5,0) {};
        \node[vertex] (7/8 1) at ($(7/8)+(120:.25)$) {};
        \node[vertex] (7/8 2) at ($(7/8)+(60:.25)$) {};
    \node[vertex] (1) at (4,0) {};
    \draw (0) to (1/2);
    \draw (1/2) to (3/4);
        \draw[dotted] (1/2) to (1/2 1);
        \draw[dotted] (1/2) to (1/2 2);
    \draw (3/4) to (7/8);
        \draw[dotted] (3/4) to (3/4 1);
        \draw[dotted] (3/4) to (3/4 2);
    \draw (7/8) to (1);
        \draw[dotted] (7/8) to (7/8 1);
        \draw[dotted] (7/8) to (7/8 2);
    \end{scope}
\end{tikzpicture}
\caption{The two generators $g_0$ and $g_1$ of the copy $H$ of Thompson's group $F$ (drawn with undirected edges, see \cref{rmk:undirected}).}
\label{fig:H:generators}
\end{figure}
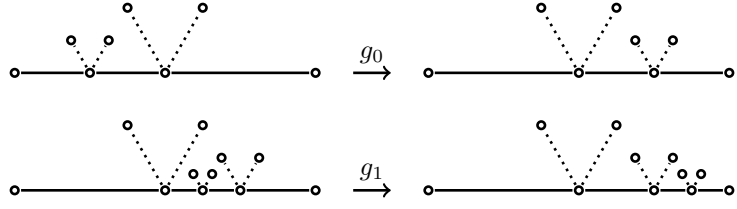

It may be noted that this copy $H$ of Thompson's group $F$ can also be produced by an application of \cite[Proposition 2.8]{belk2016rearrangement}.
In practice, up to ignoring a 180-degree flip (and, in general, the elements that switch the endpoints $\llbracket 1\bar{n} \rrbracket$ and $\llbracket \bar{n} \rrbracket$), the subgroup $H$ can be thought of as the rearrangement group obtained by forgetting about the edges $2, 3, \dots, n-1$ in both the base and the replacement graphs.

More detailed definitions of the generators $g_0$ and $g_1$ are the following (here we do not use undirected edges for the sake of completeness, so $g_0(C(11)) = C(n1)$ and $g_1(C(n1n)) = C(nn1)$ are expanded and follow the rules described in \cref{rmk:undirected}):
\begin{center}
\begin{tabular}{ l r c l }
    $g_0:$ & $1 n$ & $\mapsto$ & $1$ \\
           & $1 i$ & $\mapsto$ & $\tilde{i}$ \\
           & $1 1 1$ & $\mapsto$ & $n 1 n = n 1 \tilde{1}$ \\
           & $1 1 i$ & $\mapsto$ & $n 1 \tilde{i}$ \\
           & $1 1 n$ & $\mapsto$ & $n 1 1 = n 1 \tilde{n}$ \\
           & $i$  & $\mapsto$ & $ni$ \\
           & $n$ & $\mapsto$ & $n n$ \\
           & \\
           &
\end{tabular}
\hspace{20pt}
\begin{tabular}{ l r c l }
    $g_1:$ & $1$ & $\mapsto$ & $1$ \\
           & $i$ & $\mapsto$ & $i$ \\
           & $n 1 n$ & $\mapsto$ & $n 1 = n \tilde{n}$ \\
           & $n 1 i$ & $\mapsto$ & $n \tilde{i}$ \\
           & $n 1 1 1$ & $\mapsto$ & $n n 1 n = n n 1 \tilde{1}$ \\
           & $n 1 1 i$ & $\mapsto$ & $n n 1 \tilde{i}$ \\
           & $n 1 1 n$ & $\mapsto$ & $n n 1 1 = n n 1 \tilde{n}$ \\
           & $n i$  & $\mapsto$ & $n n i$ \\
           & $n n$ & $\mapsto$ & $n n n$
\end{tabular}
\end{center}
where $\tilde{j}$ denotes $n+1-j$ and the $i$'s always represent all possible choices among $2, \dots, n-1$.

In \cref{sec:gen}, when more about the transitivity properties of dendrite rearrangement groups will have been discussed, \cref{lem:thomps} will show that there are many more copies of Thompson's group $F$ nested inside $G_n$.

Finally, we point out that actions of Thompson's group $F$ on dendrites have been considered in literature.
For example, in \cite[Corollary 1.5]{DM18} it is proved that such actions have an orbit that consists of one or two points, and in \cite[Corollary 4.5]{FiniteOrbits} it is shown that such actions are equicontinuous on their minimal sets.
The dissertation \cite{Dendrite} is also about a Thompson-like group acting on a dendrite (namely, the Julia set for the complex map $z \to z^2 +i$, which is homeomorphic to the Wa\.zewski dendrite $D_3$):
differently from our dendrite rearrangement groups, the group studied in \cite{Dendrite} is much more $T$-like, as it is defined as a group of those homeomorphic of the unit circle that preserve a lamination induced by the Julia set.
It seems likely that this group belongs to the family of orientation-preserving dendrite rearrangement groups defined in \cref{sub:orientation}.


\section{Generation of Dendrite Rearrangement Groups}
\label{sec:gen}

In this Section we prove that Dendrite Rearrangement Groups are generated by the subgroups $H$ and $K$ from \cref{lem:thomp} and \cref{lem:K}, respectively.

\subsection{Some Transitivity Properties}

For any $k \in \mathbb{N}$, let $\boldsymbol{\mathrm{Br}^{(k)}(X)}$ be the set of all subsets of $\mathrm{Br}(X)$ consisting of exactly $k$ elements.
From \cref{lem:thomp}, it is clear that the action of $H$ on the branch points of $A_0^\mathrm{o}$ (the interior of the arc $A_0$) corresponds to the action of $F$ on the dyadic points of $(0,1)$.
In particular, since it is known that for every $k \in \mathbb{N}$ Thompson's group $F$ acts transitively on the set of $k$-tuples of dyadic points of $(0,1)$, we have the following.

\begin{lemma}
\label{lem:thomp:trans}
For every $k \in \mathbb{N}$, the subgroup $H$ acts transitively on $\mathrm{Br}^{(k)}(A_0^\mathrm{o})$.
\end{lemma}

Using this Lemma, we can prove the following useful Proposition.

\begin{proposition}
\label{prop:gen:trans}
The subgroup $\langle H, K \rangle$ of $G_n$ is transitive on $\mathrm{Br}$.
\end{proposition}

\begin{proof}
Let $p$ be a branch point and consider the path $P$ from $p_0$ to $p$ in the minimal graph expansion where $p$ appears.
We want to prove that $p$ can be mapped to $p_0$ by an element of $\langle H, K \rangle$.
We will do this by decreasing the length $k$ of $P$ in order to conclude by induction.

If $k=0$ then $p_0 = p$ and there is nothing to prove, so we can immediately suppose that $k \geq 1$.
Let $p_1$ be the vertex of $P$ that is adjacent to $p_0$.
We can suppose that $p_1$ belongs to the branch $C(1)$, as otherwise applying an element of $K$ will fix that.
Then both $p_0$ and $p_1$ are branch points of the arc $A_0$, so \cref{lem:thomp:trans} allows us to find an element of $H$ that maps $p_1$ to $p_0$ and that reduces the length of the path as needed thanks to well-known transitivity properties of Thompson's group $F$ (for a recent example of a similar application of such properties, one can see Case 2 of the proof of \cite[Theorem 6.2]{Airplane}).
\end{proof}

\subsection{A Finite Generating Set for \texorpdfstring{$G_n$}{Gn}}
\label{sub:gen}

The overall strategy of the proof of the following result is similar to that of \cite[Theorem 6.2]{Airplane}, which is about the generating set of the airplane rearrangement group $T_A$.
The similarity is no coincidence:
the close relation between the airplane rearrangement group and dendrite rearrangement groups will be discussed in \cref{sub:Air}.

\begin{theorem}
\label{thm:gen}
$G_n = \langle H, K \rangle$.
In particular, $G_n$ is finitely generated.
\end{theorem}

\begin{proof}
Let $g \in G_n$.
We need to show that $g \in \langle H, K \rangle$.
Thanks to \cref{prop:gen:trans}, up to composing $g$ with a suitable element of $\langle H, K \rangle$ we can suppose that $g$ fixes the central branch point $p_0$.
We consider the reduced graph pair diagram $(D, R, f)$ for $g$, and we proceed by induction on the number $k$ of branch points that appear as vertices in $D$.
The base case $k=1$ simply means that $g \in K$, so we can immediately suppose that $k \geq 2$.

The simplest case is when the $k$ branch points of $D$ are scattered among at least two distinct branches at $p_0$.
In this case, there is an element $h$ of $K$ such that $h g$ fixes setwise each of the $n$ branches at $p_0$.
The element $h g$ can then be decomposed as the product of its restrictions in each branch at $p_0$, and since there are at least two of these branches containing one of the $k$ branch points, each of these restrictions must contain less than $k$ branch points.
This allows conclude by applying our induction hypothesis on each of the restrictions.

Suppose instead that the $k$ branch points of $D$ all belong to the same branch at $p_0$.
We can assume that $g$ fixes setwise each of the $n$ branches at $p_0$, up to composing with a suitable element of $K$.
In this case, we will apply a similar strategy to that described at the end of the proof of \cref{prop:gen:trans} (also similar to Case 2 of the proof of \cite[Theorem 6.2]{Airplane}):
we will find $h_D,h_R \in H$ such that the reduced graph pair diagram of $h_Rgh_D^{-1}$ features less branch points than $(D,R,f)$.
Let $(v_0,v_1,\dots,v_m)$ be vertices of the longest path in $D$ that corresponds to an arc included in $A_0^\mathrm{o}$ (equivalently, this is the longest path in $D$ whose edges correspond to words $nw$ for some $w$ in the alphabet $\{1,n\}$, excluding paths that contain the vertex $\llbracket \bar{n} \rrbracket$).
Then $v_0=p_0$ and $2 \leq m \leq k$.
Consider the element $h_D$ of $H$ that maps $p_0$ to $\llbracket 1 i \bar{n} \rrbracket$ (for any $i$) and every other $v_j$ to $\llbracket n^j i \bar{n} \rrbracket$ (for any $i$), mapping rigidly between the edges between such vertices.
See \cref{fig:thm:gen} for a schematic depiction of such an element, where the $v_i$'s lie on $A_0$ in a generic position.
Let us construct $h_R$ in the same way for the range graph $R$.
Now consider the element $h_Rgh_D^{-1} \in \langle H,K \rangle$:
by the construction of $h_D$ and $h_R$, its domain and range graphs have $k$ vertices that correspond to branch points, of which $k-1$ are included in $A_0$ and the other is $\llbracket 1 i \bar{n} \rrbracket$.
Both graphs feature the subgraph with edges $11, 12, \dots, 1n$, which we denote by $\Lambda$.
Moreover, $h_Rgh_D^{-1}$ maps the cell $C(1i)$ rigidly to itself for any $i=1,\dots,n-1$.
Thus, the subgraph $\Lambda$ can be reduced in the graph pair diagram of $h_Rgh_D^{-1}$, which decreases by one the number of branch points that appear.
We can thus apply our induction hypothesis on $h_Rgh_D^{-1}$ and conclude that $g \in \langle H, K \rangle$.
\end{proof}

\begin{figure}
\centering
\begin{tikzpicture}[scale=1.333]
    \node[vertex] (0) at (0,0) {};
    \node[vertex] (v0) at (2,0) {};
        \node[below] at (v0) {$v_0$};
        \node[vertex] (v0l) at ($(v0)+(120:1)$) {};
        \node[vertex] (v0r) at ($(v0)+(60:1)$) {};
    \node[vertex] (v1) at (2.39,0) {};
        \node[below] at (v1) {$v_1$};
        \node[vertex] (v1l) at ($(v1)+(120:.225)$) {};
        \node[vertex] (v1r) at ($(v1)+(60:.225)$) {};
    \node[vertex] (v2) at (3,0) {};
        \node[below] at (v2) {$v_2$};
        \node[vertex] (v2l) at ($(v2)+(120:.5)$) {};
        \node[vertex] (v2r) at ($(v2)+(60:.5)$) {};
    \node[vertex] (v3) at (3.65,0) {};
        \node[below] at (v3) {$v_3$};
        \node[vertex] (v3l) at ($(v3)+(120:.275)$) {};
        \node[vertex] (v3r) at ($(v3)+(60:.275)$) {};
    \node[vertex] (1) at (4,0) {};
    \draw (0) to (v0);
        \draw[dotted] (v0) to (v0l);
        \draw[dotted] (v0) to (v0r);
    \draw (v0) to (v1);
        \draw[dotted] (v1) to (v1l);
        \draw[dotted] (v1) to (v1r);
    \draw (v1) to (v2);
        \draw[dotted] (v2) to (v2l);
        \draw[dotted] (v2) to (v2r);
    \draw (v2) to (v3);
        \draw[dotted] (v3) to (v3l);
        \draw[dotted] (v3) to (v3r);
    \draw (v3) to (1);
    \draw[->] (4.375,0) to node[above]{$h_D$} (4.75,0);
    \begin{scope}[xshift=5.125cm]
    \node[vertex] (0) at (0,0) {};
    \node[vertex] (w0) at (1,0) {};
        \node[below] at (w0) {$w_0$};
        \node[vertex] (w0l) at ($(w0)+(120:.5)$) {};
        \node[vertex] (w0r) at ($(w0)+(60:.5)$) {};
    \node[vertex] (w1) at (2,0) {};
        \node[below] at (w1) {$w_1$};
        \node[vertex] (w1l) at ($(w1)+(120:1)$) {};
        \node[vertex] (w1r) at ($(w1)+(60:1)$) {};
    \node[vertex] (w2) at (3,0) {};
        \node[below] at (w2) {$w_2$};
        \node[vertex] (w2l) at ($(w2)+(120:.5)$) {};
        \node[vertex] (w2r) at ($(w2)+(60:.5)$) {};
    \node[vertex] (w3) at (3.619,0) {};
        \node[below] at (w3) {$w_3$};
        \node[vertex] (w3l) at ($(w3)+(120:.25)$) {};
        \node[vertex] (w3r) at ($(w3)+(60:.25)$) {};
    \node[vertex] (1) at (4,0) {};
    \draw (0) to (w0);
        \draw[dotted] (w0) to (w0l);
        \draw[dotted] (w0) to (w0r);
    \draw (w0) to (w1);
        \draw[dotted] (w1) to (w1l);
        \draw[dotted] (w1) to (w1r);
    \draw (w1) to (w2);
        \draw[dotted] (w2) to (w2l);
        \draw[dotted] (w2) to (w2r);
    \draw (w2) to (w3);
        \draw[dotted] (w3) to (w3l);
        \draw[dotted] (w3) to (w3r);
    \draw (w3) to (1);
    \end{scope}
\end{tikzpicture}
\caption{A schematic depiction of the element $h_D$ for the proof of \cref{thm:gen}, with $m=3$. Here, $w_i$ denotes the image of $v_i$.}
\label{fig:thm:gen}
\end{figure}

Now that we know that each $G_n$ is finitely generated, it is natural to ask the following question, which we do not investigate here.

\begin{question}
What further finiteness properties do dendrite rearrangement groups have?
Are they finitely presented?
Are they $F_\infty$?
\end{question}

\subsection{Further Transitivity Properties}

Let $A$ be any rational arc (\cref{def:arc:type}) and consider the set $\mathrm{Br}(A^\mathrm{o})$ of branch points that are contained in the interior of $A$.
The next Lemma shows that $G_n$ contains a copy of Thompson's group $F$ for every arc $A = [p,q]$, where $p, q \in \mathrm{Br} \cup \mathrm{REn}$.

\begin{lemma}
\label{lem:thomps}
Given a dendrite replacement system $\mathcal{D}_n$, suppose that $A$ is a rational arc.
Then $G_n$ contain isomorphic copies $H_A$ of Thompson's group $F$ that acts on $A$ as Thompson's group $F$ does on $[0,1]$ and on $\mathrm{Br}(A^\mathrm{o})$ as $F$ does on the set of dyadic points of $(0,1)$.
In particular, $H_A$ acts transitively on $\mathrm{Br}^{(k)}(A^\mathrm{o})$.
\end{lemma}

\begin{proof}
It suffices to prove the statement for an EE-arc $A$.
Indeed, each BE-arc or BB-arc is contained in an EE-arc and corresponds to a dyadic interval under the identification of $A$ with $[0,1]$ described in \cref{rmk:arcs:dyadic}, and it is known that $F$ contains a copy of itself on every dyadic interval that is contained in $[0,1]$.

Let $A = [p,q]$ be an EE-arc.
As explained in \cref{rmk:paths}, consider the minimal expansion $E$ of $\mathcal{D}_n$ that contains the rational endpoints $p$ and $q$ as vertices.
In this graph, the arc $A$ corresponds to a path $P$ in the sense that, if $e_1, \dots, e_k$ are the edges of $P$,
then
\[ A = \big\{ \llbracket e_i \omega \rrbracket \in D_n \mid i = 1, \dots, n \text{ and } \omega \text{ is a sequence in } \{ 1, n \} \big\}. \]
By induction on the length $k$ of the path $P$ in the minimal graph where both $p$ and $q$ appear, we will find an element $g \in G_n$ such that $g(P)$ corresponds to $A_0$ in the sense described above, so that $g(A) = A_0$.
The goal will essentially be to find an element that modifies $P$ so that a reduction occurs, in order to decrease $k$.

To make this easier, we will suppose that $P$ contains the central branch point $p_0$.
We can do this without loss of generality up to conjugating by an element of $G_n$ that maps some branch point of $P$ to $p_0$, which exists thanks to \cref{prop:gen:trans}.

The base case is $k=2$, meaning that $P$ consists of two edges of the base graph.
In this case, the two edges correspond simply to distinct 1-letter words $i$ and $j$ chosen among $1, \dots, n$.
Then $g$ is the unique element of $K$ that switches $i$ with $1$ and $j$ with $n$, leaving everything else unchanged.

For $k \geq 3$, either $p$ or $q$ is not among the endpoints $\llbracket i \bar{n} \rrbracket$, and without loss of generality we will assume that it is $p$.
Denote by $r$ the vertex of $P$ that is adjacent to $p$.
Observe that $r$ must be a branch point, so $r = \llbracket x i 1 \bar{n} \rrbracket$ for some $x$ that does not end with $1$ nor with $n$ (\cref{rmk:br:ex:pts}).
Then it is not possible that $p$ is $\llbracket x \bar{n} \rrbracket$, otherwise $E$ would not be the minimal graph expansion where $p$ and $q$ appear, because the cell $C(x)$ could be reduced (\cref{fig:lem:thomps} shows why this is the case with an example).
We now find an element $g \in G_n$ that causes precisely this reduction, i.e., $g$ is such that $g(P)$ has one edge less than $P$ in the minimal graph expansion where both $g(p)$ and $g(q) = q$ appear (note that $g(q) = q$ because, as assumed at the beginning of the proof, $P$ includes the central branch point $p_0$, so $q$ is not included in any of the two branches permuted by $g$).
A simple example is depicted in \cref{fig:path:reduction};
more generally, in reference to this figure, there may be more vertices on the path joining $r$ with $g(p)$, in which case the edge between $r$ and $p$ needs to be expanded accordingly to find a graph pair diagram for $g$ that switches $p$ and $g(p)$.
Since $r = \llbracket x i 1 \bar{n} \rrbracket$, we have that $p = \llbracket x j \bar{n} \rrbracket$ for some $j$ that is not $n$ (as noted above) nor $1$ (or $p$ would not be an endpoint).
Then we choose $g$ to be the unique involution that switches the edges $xj$ and $xn$ without changing anything else.
Now, $g(P)$ admits the aforementioned reduction and we can apply our induction hypothesis on $g(p)$ and $g(q) = q$, finding an element $h \in G_n$ such that $hg(A) = A_0$ and $H^{hg}$ is the desired copy of Thompson's group $F$ acting on $A$.
\end{proof}

\begin{figure}
\centering
\begin{tikzpicture}[scale=1.333]
    \node[vertex] (c) at (0,0) {};
    \node[vertex] (l) at (-2,0) {};
    \node[vertex] (t) at (0,2) {};
    \node[vertex] (b) at (0,-2) {};
    \node[vertex] (r) at (2,0) {};
        \node[vertex] (rc) at (1,0) {};
        \node[vertex,blue] (rt) at (1,.8) {}; \draw (1,.8) node[above right,blue]{$p$};
            \node[vertex,Green] (rtc) at (1,.4) {}; \draw (1,.4) node[above right,Green]{$r$};
            \node[vertex] (rtl) at (.7,.4) {};
            \node[vertex] (rtr) at (1.3,.4) {};
        \node[vertex] (rb) at (1,-.8) {};
    \draw[black] (c) to (l);
    \draw[black] (c) to (t);
    \draw[black] (c) to (b);
    \draw[black] (rc) to (c);
    \draw[black] (rc) to (r);
    \draw[red] (rtc) to (rtl);
    \draw[red] (rtc) to (rtr);
    \draw[red] (rtc) to (rc);
    \draw[red] (rtc) to (rt);
    \draw[black] (rc) to (rb);
\end{tikzpicture}
\caption{In the proof of \cref{lem:thomps}, \textcolor{blue}{$p$} and \textcolor{Green}{$r$} cannot be as in this example or the \textcolor{red}{red edges} could be reduced.}
\label{fig:lem:thomps}
\end{figure}
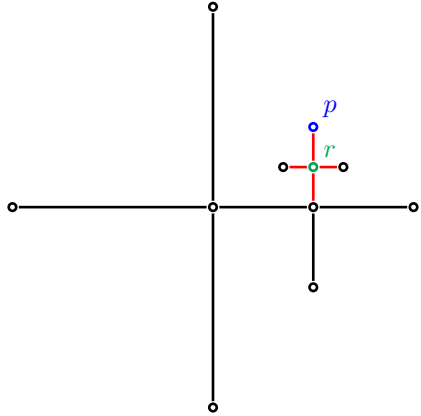

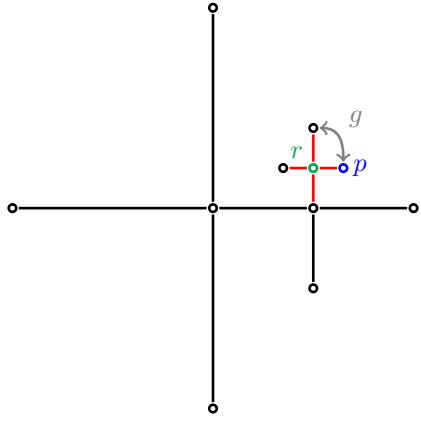
\begin{figure}
\centering
\begin{tikzpicture}[scale=1.333]
    \node[vertex] (c) at (0,0) {};
    \node[vertex] (l) at (-2,0) {};
    \node[vertex] (t) at (0,2) {};
    \node[vertex] (b) at (0,-2) {};
    \node[vertex] (r) at (2,0) {};
        \node[vertex] (rc) at (1,0) {};
        \node[vertex] (rt) at (1,.8) {};
            \node[vertex,Green] (rtc) at (1,.4) {}; \draw (1,.4) node[above left,Green]{$r$};
            \node[vertex] (rtl) at (.7,.4) {};
            \node[vertex,blue] (rtr) at (1.3,.4) {}; \draw (1.3,.4) node[right,blue]{$p$};
        \node[vertex] (rb) at (1,-.8) {};
    \draw[black] (c) to (l);
    \draw[black] (c) to (t);
    \draw[black] (c) to (b);
    \draw[black] (rc) to (c);
    \draw[black] (rc) to (r);
    \draw[red] (rtc) to (rtl);
    \draw[red] (rtc) to (rtr);
    \draw[red] (rtc) to (rc);
    \draw[red] (rtc) to (rt);
    \draw[black] (rc) to (rb);
    \draw[<->,gray] (rtr) to[out=90,in=0,looseness=1.25] node[above right]{$g$} (rt);
\end{tikzpicture}
\caption{An example of the element $g$ found in the proof of \cref{lem:thomps} that causes a reduction of the \textcolor{red}{red edges}.}
\label{fig:path:reduction}
\end{figure}

\begin{remark}
Given any two rational arcs $A_1$ and $A_2$, the subgroups $H_{A_1}$ and $H_{A_2}$ are conjugate in $G_n$ if and only if $A_1$ and $A_2$ are of the same type (in the sense of \cref{def:arc:type}).
\end{remark}


\section{Density of Dendrite Rearrangement Groups}
\label{sec:dns}

In this Section we find useful transitivity properties of the action of dendrite rearrangement groups on the set of branch points that allow to ``approximate'' any homeomorphism of $D_n$ by rearrangements.
We will use this to prove that each dendrite rearrangement group $G_n$ is dense in $\mathbb{H}_n$, where $\mathbb{H}_n$ denotes the full group of homeomorphisms of the Wa\.zewski dendrite $D_n$.
This argument is similar to that used in \cite{BasilicaDense} to show that the rearrangement group $T_B$ of the Basilica is dense in the full group of orientation-preserving homeomorphisms of the Basilica.

\subsection{Transitivity Properties of Dendrite Rearrangement Groups}

The next Proposition is an analog of \cite[Proposition 6.1]{DM19}, which goes to show that $G_n$ has many of the transitivity properties of $\mathbb{H}_n$ despite being only a countable group.
The construction of the tree $T(\mathcal{F})$ described below is heavily inspired by the one given at the beginning of Section 6 of \cite{DM19}, which in turn is not much different from the construction used in \cite[Theorem 4.14]{belk2016rearrangement} for studying the rearrangement group of the Vicsek fractal (which is homeomorphic to a dendrite, see \cref{sub:Vic}).
Similar tree approximations for dendrites are common and have appeared elsewhere in literature (see, for example, Section X.3 of \cite{continua}).

\phantomsection\label{txt:trees}
Let $\mathcal{F}$ be a finite subset of $\mathrm{Br}$ and consider the unique subdendrite $[\mathcal{F}]$ of $D_n$ that contains $\mathcal{F}$ and is minimal with respect to set inclusion.
It is homeomorphic to the geometric realization of a finite simplicial tree that has a vertex for each element of $\mathcal{F}$ (and possibly more vertices where the tree needs to branch out).
This simplicial tree is not unique, since we can choose to add vertices in the middle of any edge, so we choose $T(\mathcal{F})$ to be the minimal among these trees, and since $\mathcal{F} \subset \mathrm{Br}$ we have that each vertex of $T(\mathcal{F})$ is a branch point.
A simple example is depicted in \cref{fig:T(F)}.

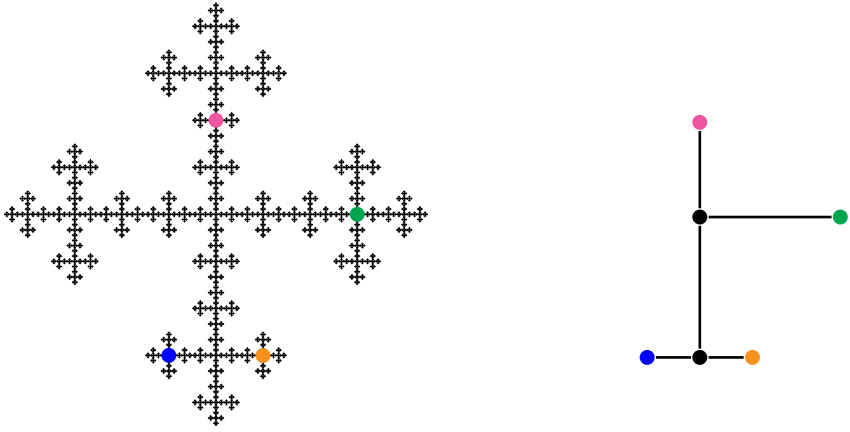
\begin{figure}
\centering
\begin{tikzpicture}
    \node[anchor=south west,inner sep=0] (image) at (0,0) {\includegraphics[width=0.448\textwidth]{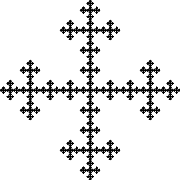}};
    \begin{scope}[x={(image.south east)},y={(image.north west)}]
        \draw[VioletRed] (.4966,.7227) node[circle,fill,inner sep=2pt]{} node[right=.2cm] {A};
        \draw[Green] (.8259,.5034) node[circle,fill,inner sep=2pt]{} node[above left]{B};
        \draw[BurntOrange] (.6062,.174) node[circle,fill,inner sep=2pt]{} node[above right]{C};
        \draw[blue] (.3868,.1744) node[circle,fill,inner sep=2pt]{} {} node[above left]{D};
    \end{scope}
\end{tikzpicture}
\hspace{15pt}
\begin{tikzpicture}[scale=2.8]
    \draw[white] (-1,-1) -- (-1,1) -- (1,1) -- (1,-1) -- (-1,-1);
    \draw[VioletRed] (0,.45) node[circle,fill,inner sep=2pt] (purple) {} node[right] {A};
    \draw[Green] (.667,0) node[circle,fill,inner sep=2pt] (green) {} node[above left]{B};
    \draw[BurntOrange] (.25,-.667) node[circle,fill,inner sep=2pt] (red) {} node[above right]{C};
    \draw[blue] (-.25,-.667) node[circle,fill,inner sep=2pt] (blue) {} node[above left]{D};
    \node[black,circle,fill,inner sep=2pt] (center) at (0,0) {};
    \node[black,circle,fill,inner sep=2pt] (bottom) at (0,-.667) {};
    \draw (center) to (green);
    \draw (center) to (purple);
    \draw (center) to (bottom);
    \draw (bottom) to (red);
    \draw (bottom) to (blue);
\end{tikzpicture}
\caption{An example of the tree $T(\mathcal{F})$ depicted beside the finite subset $F \subset \mathrm{Br}$, for $n=4$.}
\label{fig:T(F)}
\end{figure}

\begin{proposition}
\label{prop:trans}
Given two finite subsets $\mathcal{F}_1$ and $\mathcal{F}_2$ of $\mathrm{Br}$, any graph isomorphism between $T(\mathcal{F}_1)$ and $T(\mathcal{F}_2)$ can be extended to an element of $G_n$.
\end{proposition}

\begin{proof}
Given a graph isomorphism $f \colon T(\mathcal{F}_1) \to T(\mathcal{F}_2)$, we proceed by induction on the number $m$ of vertices of $T(\mathcal{F}_1)$ and $T(\mathcal{F}_2)$.
If $m = 1$, then $\mathcal{F}_1 = \{p\}$, $\mathcal{F}_2 = \{q\}$ and \cref{prop:gen:trans} provides an element of $G$ that maps $p$ to $q$.

Suppose that $m \geq 2$.
Consider the trees $T(\mathcal{F}_1)^*$ and $T(\mathcal{F}_2)^*$
obtained by removing any leaf $l$ from $T(\mathcal{F}_1)$ and its image $f(l)$ from $T(\mathcal{F}_2)$, respectively.
Observe that $T(\mathcal{F}_1)^* = T(V_1 \setminus \{l\})$ and $T(\mathcal{F}_2)^* = T(V_2 \setminus \{f(l)\})$, where each $V_i$ is the set of vertices of the respective $T(\mathcal{F}_i)$, each corresponding to a branch point.
It is worth noting that $T(V_1 \setminus \{l\})$ and $T(V_2 \setminus \{f(l)\})$ are not always the same as $T(\mathcal{F}_1 \setminus \{l\})$ and $T(\mathcal{F}_2 \setminus \{f(l)\})$, since the two former trees might need to include a vertex that would be unnecessary in the latter.
In \cref{fig:T(F)} the removal from $\mathcal{F}$ of any of its points also causes the removal of one of the two additional vertices of $T(\mathcal{F})$;
for example removing the green vertex B also cause the removal of the black vertex that is adjacent to it.
By considering $T(V_1 \setminus \{l\})$ and $T(V_2 \setminus \{f(l)\})$ we are sure that only one vertex is removed.

We can then apply our induction hypothesis, which produces an element $g^* \in G$ that extends the graph isomorphism $f^* \colon T(\mathcal{F}_1)^* \to T(\mathcal{F}_2)^*$ induced (uniquely) from $f$.
It only remains to adjust the element $g^*$ so that it also maps $l$ to $f(l)$.
Let $v$ be unique vertex of $T(\mathcal{F}_2)$ that is adjacent to the leaf $f(l)$.
Note that there is at least one branch $B_1$ (respectively, $B_2$)  at $v$ that includes $v$ and $f(l)$ (respectively, $g^*(l)$) but does not include any other vertex of $T(\mathcal{F}_2)$;
it is possible that $B_1$ and $B_2$ are the same branch.
By \cref{lem:perm} we can choose an element $h_1$ of $G_n$ that switches the branches $B_1$ and $B_2$ and fixes every other point (in case $B_1 = B_2$, we choose $h_1$ to be the identity).
Then we have that $h_1$ fixes every vertex of $T(\mathcal{F}_2)^*$ and that $\tilde{l} = h_1 g^* (l)$ belongs to the same branch $B_1$ that contains $f(l)$.
Now that we have two branch points $\tilde{l} = h_1 g^* (l)$ and $f(l)$ belonging to the same branch $B_1$, we can use \cref{lem:thomps} to find an element $h_2 \in H_{A}$ that maps $h_1 g^* (l)$ to $f(l)$, where $A$ is any rational arc included in $B_1$ whose interior includes $f(l)$ and $\tilde{l}$.
Clearly $h_2$ must fix $v \in B_1$, and since every other vertex of $T(\mathcal{F}_2)^*$ does not belong to $B_1$ we have that the element $g = h_2 h_1 g^*$ maps each vertex of $T(\mathcal{F}_1)$ to its image under the action of $f$, as required.
\end{proof}

This transitivity property of $G_n$ on the set of branch points is the same that the full homeomorphism group $\mathbb{H}_n$ has by \cite[Proposition 6.1]{DM19}, so \cref{prop:trans} is essentially saying that every dendrite rearrangement group is as transitive as a subgroup of $\mathbb{H}_n$ can be on the set of branch points.

As an immediate consequence of \cref{prop:trans}, we have the two following facts (which are analogous to Corollaries 6.5 and 6.7 of \cite{DM19}, respectively).

\begin{corollary}
The stabilizer of a branch point is a maximal proper subgroup of $G_n$.
\end{corollary}

\begin{corollary}
The action of $G_n$ on the set $\mathrm{Br}$ is oligomorphic.
\end{corollary}

Recall that a $G$-action on a set $X$ is \textbf{oligomorphic} if, for every $m \in \mathbb{N}$, the diagonal $G$-action on the set of $m$-tuples of $X$ has finitely many orbits.
For more information about oligomorphic group actions, see \cite{Oligomorphic}.

\subsection{Density in the Full Group of Homeomorphisms}
\label{sub:dns}

Here we prove that \cref{prop:trans} has the following consequence.

\begin{theorem}
\label{thm:dense}
For every $n \geq 3$, the rearrangement group $G_n$ for the dendrite replacement system $\mathcal{D}_n$ is dense in $\mathbb{H}_n = \mathrm{Homeo}(D_n)$.
\end{theorem}

As done in \cite{DM19}, here we endow the group $\mathbb{H}_n$ with its topology of uniform convergence, which makes it a Polish group (i.e., a separable and completely metrizable topological group).
By \cite[Proposition 2.4]{DM19}, this topology of $\mathbb{H}_n$ is inherited by the pointwise convergence topology of $\mathrm{Sym}\big(\mathrm{Br}\big)$, so in order to prove \cref{thm:dense} it suffices to show that the permutations of $\mathrm{Br}$ induced by elements of $\mathbb{H}_n$ can be ``approximated'' by those of $G_n$.
Thus, \cref{thm:dense} is equivalent to the following.

\begin{claim}
\label{clm:dns}
Let $\phi \in \mathbb{H}_n$.
For every $k \geq 1$ and for each $p_1, \dots, p_k \in \mathrm{Br}$ there exists a rearrangement $g_k \in G_n$ such that $g_k(p_i) = \phi(p_i)$ for all $i=1,\dots,k$.
\end{claim}

\begin{proof}
Consider a homeomorphism $\phi$ of $D_n$ such that $\phi(p_i) = q_i$ for branch points $p_1, \dots, p_k, q_1, \dots, q_k$.
Let $\mathcal{F}_1 = \{p_1, \dots, p_k\}$ and $\mathcal{F}_2 = \{q_1, \dots, q_k\}$ and consider the trees $T(\mathcal{F}_1)$ and $T(\mathcal{F}_2)$ obtained as described at \cpageref{txt:trees}.
Consider the unique minimal subdendrites $[\mathcal{F}_1]$ and $[\mathcal{F}_2]$ of $D_n$ that contain $\mathcal{F}_1$ and $\mathcal{F}_2$, respectively;
these are naturally homeomorphic to the geometric realizations of $T(\mathcal{F}_1)$ and $T(\mathcal{F}_2)$.
The restriction of $\phi$ to $[\mathcal{F}_1]$ is a homeomorphism $\phi|_{[\mathcal{F}_1]} \colon [\mathcal{F}_1] \to [\mathcal{F}_2]$ that maps each vertex $p_i$ of $T(\mathcal{F}_1)$ to the vertex $q_i$ of $T(\mathcal{F}_2)$, so it restricts to a graph isomorphism $T(\mathcal{F}_1) \to T(\mathcal{F}_2)$ defined by the mapping $p_i \mapsto q_i$.
By \cref{prop:trans}, there exists a rearrangement $g \in G_n$ that extends the graph isomorphism $T(\mathcal{F}_1) \to T(\mathcal{F}_2)$, meaning that it maps each $p_i$ to $q_i$, as needed.
\end{proof}

In the next Section, with \cref{prop:comm:trans,cor:comm:dns}, using these very same arguments, we will see that the commutator subgroup $[G_n, G_n]$ shares strong transitivity properties with $G_n$ and thus it is also dense in $\mathbb{H}_n$ despite being significantly ``smaller'' than $G_n$ (it has infinite index by \cref{thm:commutator}).

Observe that being dense in an ambient group is not new behaviour in the world of Thompson groups.
Indeed, \cite[Corollary A5.8]{PL-homeo} and \cite[Proposition 4.3]{Tdense} show that Thompson groups $F$ and $T$ are dense in the groups $\mathrm{Homeo}^+([0,1])$ and $\mathrm{Homeo}^+(S^1)$ of orientation-preserving homeomorphisms of the unit interval and the unit circle, respectively.
Also, Thompson's group $V$ is dense in both the group $\mathrm{AAut}(\mathcal{T}_2)$ of almost automorphisms of the binary tree (noted in \cite{Vdense}) and also in the group $\mathrm{Homeo}(\mathfrak{C})$ of homeomorphisms of the Cantor space with its Polish compact-open topology (see \cref{rmk:V:dns} below).
Notably, the very recent work \cite{BasilicaDense} shows that the rearrangement group of the Basilica is dense in the group of all orientation-preserving homeomorphisms of the Basilica.
Later in \cref{sub:Air} we will see that the Airplane rearrangement group is also dense in the group of orientation-preserving homeomorphisms of a dendrite.
All of these results naturally prompt the following question, which was also asked in \cite[Final remarks (b)]{BasilicaDense}.

\begin{question}
When is a rearrangement group of a topological space $X$ dense in the full group of homeomorphisms of $X$?
\end{question}

This question might be related to recent findings and questions about quasi-symmetries of fractals. Namely, \cite{BasilicaQuasiSymmetries} shows that the Basilica rearrangement group $T_B$ is in a certain sense dense in the group of planar quasi-symmetries of the Basilica Julia set, and \cite{QuasiSymmetries} asks if it is always possible to find a finitely generated dense subgroup of the orientation-preserving group of homeomorphisms.

We believe that this question is interesting for practical reasons.
Finding a Thompson-like group $G$ that is dense in an interesting uncountable group $\mathbb{H}$ essentially allows to approximate the ``non-computable'' group $\mathbb{H}$ with the simpler group $G$ that can probably be easily handled by a computer.

\begin{remark}
\label{rmk:V:dns}
To the best of the author's knowledge, there is no proof in literature that $V$ is dense in the group $\mathrm{Homeo}(\mathfrak{C})$ of homeomorphisms of the Cantor set with its compact-open topology.
Thus, although we do not claim the novelty of this result, we provide a quick sketch of proof here for the sake of completeness.

First, observe that $\mathrm{Homeo}(\mathfrak{C})$ with its compact-open topology is the same as the group of automorphisms of the Boolean algebra of clopen subsets of the Cantor space $\mathfrak{C}$ with its pointwise convergence topology (this was noted, for example, at the beginning of Section 2 of \cite{TDLCCantor}).
Then it suffices to show that, for every homeomorphism $\phi$ of $\mathfrak{C}$ and for each finite set $\mathcal{S}$ of clopen subsets of $\mathfrak{C}$, there exists an element $g$ of Thompson's group $V$ such that $g(A) = \phi(A)$ for all $A \in \mathcal{S}$.
Showing this is straightforward:
if we view the Cantor space $\mathfrak{C}$ as the set of infinite words in the alphabet $\{0,1\}$, then any clopen set is the disjoint union of finitely many cones (recall that a cone is a subset of $\mathfrak{C}$ consisting of all of those sequences that have a certain common prefix);
if we refine the set $\mathcal{S}$ of clopen subsets of $\mathfrak{C}$ to a set $\mathcal{S}^*$ of cones, then it is immediate to build an element of $V$ that maps each cone of $\mathcal{S}^*$ canonically to its image under the homeomorphism $\phi$.
\end{remark}


\section{The Commutator Subgroup}
\label{sec:comm}

In this Section we first show that the abelianization of every dendrite rearrangement group is $\mathbb{Z}_2 \oplus \mathbb{Z}$ (\cref{sub:abel}), then we prove that the commutator subgroup $[G_n, G_n]$ is dense in the full homeomorphism group $\mathbb{H}_n$ and that, with the possible exception of $n=3$, it is simple.

\subsection{The Abelianization of Dendrite rearrangement Groups}
\label{sub:abel}

In this Subsection we describe two families of maps:
the local parity maps $\pi_p$ ($p$ a branch point) and the local endpoint derivatives $\partial_q$ ($q$ an endpoint), both satisfying a chain rule that will allow us to define global versions of these maps.
Together, the parity map and the endpoint derivative will allow us to describe the elements of the commutator subgroup $[G_n, G_n]$ of a dendrite rearrangement group and to compute the abelianization of $G_n$ by applying Schreier's Lemma (which was also used in \cite{Belk_2015} for the computing abelianization of the Basilica rearrangement group $T_B$).

\subsubsection{The Parity Map}
\label{sub:parity}

The branches at a given branch point $p$ are naturally enumerated by $1, \dots, n$.
More precisely, there is a unique finite word $x$ for which $p = \llbracket x i 1 \overline{n} \rrbracket$;
each branch at $p$ intersects precisely one of the cells $\llbracket x i \rrbracket$ (and thus every cell whose address starts with $xi$):
this induces an enumeration of the branches at $p$ (i.e., a bijection with $\{1, \dots, n\}$).

A homeomorphism $g$ of $D_n$ induces a bijection between the branches at $p$ and those at $g(p)$.
This in turn induces a permutation $\sigma_{p,g} \in \mathrm{Sym}(n)$ according to the enumeration of the branches at $p$ and $g(p)$.
We define the \textbf{local parity map} at the branch point $p \in \mathrm{Br}$ as the mapping
\[ \pi_p \colon G_n \to \mathbb{Z}_2 \]
such that $\pi_p(g) = 0$ if $\sigma_{p,g}$ is an even permutation and $\pi_p(g) = 1$ if $\sigma_{p,g}$ is an odd permutation.

Note that the local parity maps satisfy the following chain rule:
\[ \pi_p (g h) = \pi_{h(p)} (g) + \pi_p (h). \]

Observe that, whenever a branch point $p$ does not appear in a graph pair diagram of a rearrangement $g$, one has $\pi_p(g)=0$ because canonical homeomorphisms of a cell $C(x)$ (which were defined at the beginning of \cref{sub:rearrangements}) preserve the enumeration of branches around the branch point $\llbracket x i 1 \bar{n} \rrbracket$.
In particular, this implies that the definition does not depend on the choice of graph pair diagram chosen to represent $g$.
This, together with the previously described chain rule and the fact that every rearrangement induces a permutation of $\mathrm{Br}$, allow us to define the (global) \textbf{parity map $\boldsymbol{\Pi}$} as
\[ \Pi \colon G_n \to \mathbb{Z}_2, \; g \mapsto \sum_{p \in \mathrm{Br}} \pi_p(g). \]
Because of the chain rule identity, it is easy to see that $\Pi$ is a group morphism.

\subsubsection{The Endpoint Derivative}

We define the \textbf{local endpoint derivative} at the rational endpoint $q \in \mathrm{REn}$ as the mapping
\[ \partial_q \colon G_n \to \mathbb{Z} \]
defined as follows.
If $(D,R,f)$ is a graph pair diagram for $g$, there is a unique edge $x$ of $D$ that is incident on $q$, and a unique edge $y = f(x)$ of $R$ that is incident on $f(q)$.
Denote by $L_D(x)$ and $L_R(y)$ the number of $n$'s at the end of $x$ and $y$, without considering the first letter in case the word is $n n \dots n$ (the first letter is special in the sense that it denotes an edge of the base graph instead of the replacement graph).
Then
\[ \partial_q(g) = L_D(x) - L_R(y). \]
This definition does not depend on the choice of graph pair diagram chosen to represent $g$, since expanding $x$ in $D$ always causes the expansion of $y$ in $R$, and both the resulting lengths $L_D$ and $L_R$ increase by $1$, leaving $\partial_q$ ultimately unchanged.

The intuition behind this definition is that $L_D(x)$ (or $L_R(y)$) represent how small of a portion of a ``straight branch'' $x$ (or $y$) is, when drawing $z1$ and $zn$ straight and aligned with $z$ when expanding an edge $z$.
In this sense, $\partial_q$ looks like a logarithmic derivative in a small enough neighbourhood of the endpoint $q$.

Note that, as the local parity maps do, the endpoint derivatives also satisfy the following chain rule:
\[ \partial_q (g h) = \partial_{h(q)} (g) + \partial_q (h). \]

Observe that, whenever a rational endpoint $q$ does not appear in a graph pair diagram of a rearrangement $g$, one has $\partial_q(g)=0$.
Then, for the same reasons used for the parity map, we can define the (global) \textbf{endpoint derivative $\boldsymbol{\Delta}$} as
\[ \Delta \colon G_n \to \mathbb{Z}, \; g \mapsto \sum_{q \in \mathrm{REn}} \partial_q(g). \]
Because of the chain rule identity, $\Delta$ is a group morphism.

Before proceeding with the computation of the abelianization of $G_n$, we would like to mention that this idea of an endpoint derivative was previously used to characterize the commutator subgroup $[T_A, T_A]$ of the Airplane rearrangement group.
This in turn is related to the fact that the commutator subgroup of Thompson's group $F$ consists of the elements of $F$ that act trivially around the neighbourhoods, i.e., that preserve the numbers of $0$'s (respectively $1$'s) of the leftmost (respectively rightmost) edge of a graph pair diagram.
There was no parity map for $T_A$ essentially because the subgroup of $T_A$ that corresponds to $K \leq D_n$ is Thompson's group $T$ (see \cref{sub:Air}), which is simple.

\subsubsection{Computing the Abelianization}

Consider the following map:
\[ \Phi = \Pi \times \Delta \colon G_n \to \mathbb{Z}_2 \oplus \mathbb{Z}, \; g \mapsto \big( \Pi(g), \Delta(g) \big). \]
Clearly $\Phi$ is a group morphism, and it is not hard to see how it behaves on the generators $\{g_0, g_1, \tau_2, \dots, \tau_n \}$ of $G_n$, where each $\tau_i$ is the element of $K$ that corresponds to the permutation $(1 i)$ and $g_0$ and $g_1$ were represented in \cref{fig:H:generators} and generate $H$.
\begin{align*}
    \Phi(g_0) = (0,0), \; \Phi(g_1) = (0,-1),\\
    \Phi(\tau_i) = (1,0), \; \forall i = 2, \dots, n.
\end{align*}
The parity map on $g_0$ and $g_1$ is $0$ because both of these elements feature exactly two branch points with the same non-trivial permutation (for $g_0$ these are $\llbracket i \overline{n} \rrbracket$ and $\llbracket 1 1 i \overline{n} \rrbracket$ while the permutation is trivial at $\llbracket 1 i \overline{n} \rrbracket$, while for $g_1$ they are $\llbracket n i \overline{n} \rrbracket$ and $\llbracket n 1 1 i \overline{n} \rrbracket$ while the permutation is trivial at $\llbracket i \overline{n} \rrbracket$ and $\llbracket n i \overline{n} \rrbracket$).
In particular, from these computations we deduce that $\Phi$ is surjective, because for all $\zeta \in \{0,1\}$ and $z \in \mathbb{Z}$ one has $\Phi(\tau_2^\zeta g_1^z) = (\zeta, -z)$.
If we let $K = \mathrm{Ker}(\Phi)$, then $G_n / K \simeq \mathbb{Z}_2 \oplus \mathbb{Z}$, so $[G_n, G_n] \leq K$, so our goal is to show that $K \leq [G_n, G_n]$.

Using the following set of transversals for the quotient $G_n / K$
\[ \left\{ \tau_2^{\zeta} g_1^{-z} \mid \zeta \in \{0,1\}, z \in \mathbb{Z} \right\} \]
and the generating set $\{g_0, g_1, \tau_2 \dots, \tau_n \}$ for $G_n$ (\cref{thm:gen}), together with the application of the identities $\tau_2 g_1 = g_1 \tau_2$ and $\tau_i = \tau_i^{-1}$, an application of Schreier's Lemma (see, for example, \cite[Lemma 4.2.1]{MR1970241}) produces the following generating set for $K$:
\begin{equation}
\label{eq:commutator}
    \left\{
    g_1^{-z} \tau_2 \tau_i g_1^z, \;
    \tau_2^\zeta g_1^{-z} g_0 g_1^z \tau_2^\zeta \,
    \mid \, \zeta \in \{0,1\}, z \in \mathbb{Z}, i = 3, \dots, n
    \right\}.
\end{equation}

This generating set is included in the normal closure of $\{ \tau_2 \tau_i, g_0 \mid i = 3, \dots, n \}$ in $G_n$ and the commutator subgroup is normal in $G_n$, so it suffices to show that $[G_n, G_n]$ contains $\tau_2 \tau_3, \dots, \tau_2 \tau_n$ and $g_0$ in order to be able to conclude that $K \leq [G_n, G_n]$.
For each $i \in \{3, \dots, n\}$ the element $\tau_2 \tau_i$ is an even permutation, so it belongs to $[K, K] \leq [G_n, G_n]$.
As for $g_0$, using the fact that $[H,H] \simeq [F,F]$ is the subgroup of the elements of $H$ that act trivially around the endpoints of $C(1)$ and $C(n)$, a direct computation shows that
\begin{equation}
\label{eq:g0}
g_0 = h [\tau_n, g_1] \text{ for an element } h \in [H,H],
\end{equation}
which proves that $g_0 \in [G_n, G_n]$ and ultimately that $[G_n, G_n]$ is the kernel of $\Phi$.

To sum up the results shown so far about the commutator subgroup:

\begin{theorem}
\label{thm:commutator}
The commutator subgroup $[G_n, G_n]$ of a dendrite rearrangement group $G_n$ is the kernel of the surjective morphism $\Phi = \Pi \times \Delta: G_n \to \mathbb{Z}_2 \oplus \mathbb{Z}$.
In particular, the abelianization of $G_n$ is $\mathbb{Z}_2 \oplus \mathbb{Z}$.
Moreover, $[G_n, G_n]$ is generated by the (infinite) set exhibited above in \cref{eq:commutator}.
\end{theorem}

\begin{remark}
\label{rmk:comm:gen}
Using the identities $\tau_i g_1 = g_1 \tau_i$ for $i = 2, \dots, n-1$ and $[\tau_n, g_1^z] = [\tau_n, g_1]^z$ (which can both be verified with quick computations), the part $g_1^{-z} \tau_2 \tau_i g_1^z$ of the generating set in \cref{eq:commutator} can be replaced by the finite subset
\[ \{ \tau_2 \tau_i, [\tau_n, g_1] \mid i = 3, \dots, n \} \]
This produces the following generating set for $[G_n, G_n]$, which will be useful at the end of the proof of \cref{thm:comm:simple}:
\[ \{ \tau_2 \tau_i, \; [\tau_n, g_1], \; \tau_2^\zeta g_1^{-z} g_0 g_1^z \tau_2^\zeta \, \mid \, \zeta \in \{0,1\}, z \in \mathbb{Z}, i = 3, \dots, n \}. \]
This infinite generating set will be refined to a finite generating set in \cref{thm:comm:fg} at the end of this Section.
\end{remark}

\subsection{Simplicity of the Commutator Subgroup}

An analogue of \cref{prop:trans} also applies to the commutator subgroup $[G_n, G_n]$.
In order to prove this, we need the two following notions.
If a group $G$ acts on a space $X$, the \textbf{support} of an element $g \in G$ is the set of points that are not fixed by $g$.
If $U \subseteq X$ then the \textbf{rigid stabilizer} of $U$ in $G$ is the subgroup of those elements whose support is included in $U$, i.e., the elements that act trivially outside of $U$.
In symbols, this is
\[ \boldsymbol{\mathrm{Rist}_G(U)} = \{ g \in G \mid g(x) = x, \, \forall x \in X \setminus U \} \leq G. \]
The notion of rigid stabilizers will also be essential in \cref{thm:comm:simple}.

\begin{proposition}
\label{prop:comm:trans}
Given two finite subsets $\mathcal{F}_1$ and $\mathcal{F}_2$ of $\mathrm{Br}$, any graph isomorphism between $T(\mathcal{F}_1)$ and $T(\mathcal{F}_2)$ can be extended to an element of $[G_n,G_n]$.
In particular, $[G_n, G_n]$ acts transitively on $\mathrm{Br}(D_n)$.
\end{proposition}

\begin{proof}
There is no need to repeat the same argument that was used in \cref{prop:trans}.
Instead, one can reason in the following manner.
Because of the aforementioned \cref{prop:trans}, given $\mathcal{F}_1, \mathcal{F}_2$ and a graph isomorphism $f \colon T(\mathcal{F}_1) \to T(\mathcal{F}_2)$ as in the statement, there exists a $g \in G_n$ that extends $f$.
Now, suppose that $\Phi(g) = (\zeta, z) \in \mathbb{Z}_2 \oplus \mathbb{Z}$.
Since $[G_n, G_n]$ is none other than the kernel of $\Phi$ (\cref{thm:commutator}), our goal is to produce an ``adjustment'' of the element $g$ that nullifies $\Phi$ and is the same as $g$ when restricted to $\mathcal{F}_1$.

Since the minimal subdendrite $[\mathcal{F}_2]$ of $D_n$ that contains $\mathcal{F}_2$ cannot be the entire $D_n$, there exists a branch $B^*$ that is fully contained in the complement of $[\mathcal{F}_2]$.
It is not hard to see that the restriction of $\Phi$ on the rigid stabilizer of any branch of $D_n$ is surjective, so in particular one can find an element $h \in \mathrm{Rist}_{G_n}(B^*)$ such that $\Phi(h) = (\zeta, -z)$.
It follows that the element $h g$ belongs to the commutator subgroup of $G_n$ and is the same as $g$ when restricted to $\mathcal{F}_1$, meaning that it extends the graph isomorphism $f: T(\mathcal{F}_1) \to T(\mathcal{F}_2)$ as desired.
\end{proof}

Using the exact same argument of \cref{clm:dns}, this Proposition allows to immediately prove that the commutator subgroup $[G_n, G_n]$ is also dense in the full group of homeomorphisms of $D_n$.

\begin{corollary}
\label{cor:comm:dns}
The commutator subgroup of the dendrite rearrangement group of $D_n$ is dense in the full group of homeomorphisms of $D_n$.
\end{corollary}

As another consequence of \cref{prop:comm:trans}, we have:

\begin{corollary}
\label{cor:comm:trans}
$[G_n, G_n]$ acts transitively on the set of branches of $D_n$.
\end{corollary}

\begin{proof}
Consider two branches $B_1$ and $B_2$.
If $p_1$ and $p_2$ are the branch points at which they branch out, respectively, then by \cref{prop:comm:trans} there exist elements $h_1,h_2 \in [G_n, G_n]$ such that $h_1(p_1)=p_0$ and $h_2(p_0)=p_2$.
Since both $h_1(B_1)$ and $h_2^{-1}(B_2)$ are branches at $p_0$, there exists some element $\sigma$ of the alternating subgroup $[K, K] \leq [G_n, G_n]$ (which is transitive on the set of branches at $p_0$, see \cref{lem:perm}) such that $h_2^{-1} \sigma h_1 (B_1) = B_2$.
Thus, the element $h_2^{-1} \sigma h_1$ maps $B_1$ to $B_2$ and belongs to $[G_n, G_n]$, as needed.
\end{proof}

This last Corollary allows to prove the simplicity of $[G_n, G_n]$ for all $n \geq 4$.

\begin{theorem}
\label{thm:comm:simple}
For each $n\geq4$, the commutator subgroup of the dendrite rearrangement group $G_n$ is simple.
\end{theorem}

\begin{proof}
This proof shares its structure with the proofs of the simplicity of the commutator subgroups of the Basilica and Airplane limit spaces (from \cite{Belk_2015} and \cite{Airplane}, respectively), which in turn were inspired by the so-called \textit{Epstein's double commutator trick} from \cite{Epstein}.

Given a non-trivial normal subgroup $N$ of $[G_n, G_n]$, we wish to prove that $N = [G_n, G_n]$.
Let $g \in N \setminus \{1\}$.
Since $g$ is non-trivial, there must be an open subset $U$ of $D_n$ such that $U \cap g(U) = \emptyset$.
Every open subset of $D_n$ contains a cell (see \cite[Lemma 2.2]{IG}) and every cell contains a branch;
thus, there exists a branch $B$ based at some branch point $p$ that is included in $U$, and so $B \cap g(B) = \emptyset$.

Let $h_1, h_2 \in \mathrm{Rist}_{[G_n, G_n]}(B)$.
Note that $g^{-1} h_1 g \in \mathrm{Rist}_{[G_n, G_n]} \big( g^{-1}(B) \big)$, so
\[ [h_1, g] = h_1^{-1} g^{-1} h_1 g \in \mathrm{Rist}_{[G_n, G_n]} \big( B \cup g^{-1}(B) \big) \]
and $[h_1, g] |_B = h_1^{-1} |_B$.
Thus, since $h_2 \in \mathrm{Rist}_{[G_n, G_n]}(B)$, direct computations of $\big[ [h_1, g], h_2 \big]$ restricted on $B$, $g^{-1}(B)$ and elsewhere show that globally
\[ \big[ [h_1, g]^{-1}, h_2 \big] =[h_1,g] h_2^{-1} [h_1,g]^{-1} h_2 = h_1^{-1} h_2^{-1} h_1 h_2 = [h_1, h_2]. \]
Now, since $[h_1, g] \in N$ (because $g \in N \trianglelefteq [G_n, G_n] \ni h_1$) and $h_2 \in [G_n, G_n]$, the double commutator of the previous identity belongs to $N$, and so we have that
\[ [h_1, h_2] \in N \text{ for all } h_1, h_2 \in \mathrm{Rist}_{[G_n, G_n]}(B). \]

We do not have much control over the choice of $B$, but this can be overcome with the aid of \cref{cor:comm:trans}:
for any branch $B^*$, there exists some $h \in [G_n, G_n]$ such that $h(B^*) = B$, and conjugating the previous identity by $h$ yields
\begin{equation}
\label{eq:comm:simple}
[h_1, h_2] \in N \text{ for all } h_1, h_2 \in \mathrm{Rist}_{[G_n, G_n]}(B^*).
\end{equation}
As is shown below, this identity allows us to show that $N$ contains the generating set for $[G_n, G_n]$ exhibited in \cref{rmk:comm:gen}, which will conclude this proof.

Let us first show that $\tau_2 \tau_i \in N$ for all $i = 3, \dots, n$ because $N$ contains the entire subgroup $[K, K]$.
Each $3$-cycle $(i j k)$ can be written as $[(ik),(jk)]$.
Since $(ik)$ and $(jk)$ do not belong to $[G_n,G_n]$, we need to replace them by some $h_1 = (ik)h$ and $h_2=(jk)h$ in $\mathrm{Rist}_{[G_n,G_n]}(B^*)$ (for any branch $B^*$) such that $(ijk) = [(ik), (jk)] = [h_1,h_2]$, in order to apply \cref{eq:comm:simple}.
Using the fact that $n\geq4$, such an $h$ can be found, for example, as follows.
Let $p' = \llbracket kl1\overline{n} \rrbracket$ (for any $l=1,\dots,n$), which is the central branch point of the cell $C(k)$.
Let $h$ be the permutation at $p'$ that only switches the cells $C(k (n-1))$ and $C(k n)$, which is an odd permutation whose support is included in $C^\mathrm{o}(k)$.
Then $h_1 = (ik)h$ and $h_2 = (jk)h$ are compositions of odd permutations, thus $h_1, h_2 \in [G_n,G_n]$.
Since $n\geq4$, both $h_1$ and $h_2$ act trivially on some branch $\tilde{B}$ at $p_0$ that is different from $i, j$ and $k$.
(This last sentence does not work when $n=3$, see \cref{rmk:3:comm:simple,qst:3:comm:simple} below.)
If we fix any branch point $p^*$ of $\tilde{B}$, the elements $h_1$ and $h_2$ belong to $\mathrm{Rist}_{[G_n,G_n]}(B^*)$, where $B^*$ is the unique branch at $p^*$ that includes $p_0$.
Thus, one concludes by \cref{eq:comm:simple}.

To show that $[\tau_n, g_1]$ belongs to $N$, we need to rewrite it as $[\tau_n \rho, g_1 f]$ for suitable elements $\rho$ and $f$ so that we can apply \cref{eq:comm:simple}.
To find $\rho$ and $f$ such that $\rho\tau_n$ and $g_1f$ belong to $[G_n,G_n]$, we can employ a strategy similar to that used in \cref{prop:comm:trans}.
For example, let $\rho$ be any odd permutation with support included in $C(212)$ (for instance, the element that switches $C(2122)$ and $C(2123)$) and let $f$ be the element that acts on $C(2n)$ as $g_1$ doe on $C(n)$.
Then $\tau_n\rho$ and $g_1f$ belong to $\mathrm{Ker}(\Phi)=[G_n,G_n]$.
Moreover, the supports of $\rho$ and $f$ are disjoint and each of them has trivial intersection with the supports of $\tau_n$ and $g_1$, so $[\tau_n, g_1] = [\tau_n \rho, g_1 f]$.
Since the union $U$ of the supports of $\rho$, $f$, $\tau_n$ and $g_1$ is a proper closed subset of $D_n$, it is included in some branch $B$.
It follows that both $\tau_n \rho$ and $g_1 f$ belong to $\mathrm{Rist}_{[G_n, G_n]}(B)$, so we can conclude that $[\tau_n, g_1] = [\tau_n \rho, g_1 f] \in N$ by \cref{eq:comm:simple}.

Now that we know that $N$ contains $[\tau_n, g_1]$, because of \cref{eq:g0} it suffices to prove that $[H,H] \leq N$ in order to show that $g_0 \in N$ (recall that $H$ is the Thompson subgroup described in \cref{sub:thomp}).
This is done by recalling that the commutator subgroup of Thompson's group $F$ is the subgroup of $F$ that acts trivially around the endpoints of the unit interval, and that it is simple and non-abelian, so the commutator of $[H,H]$ is $[H,H]$ itself.
Then an arbitrary element of $[H,H]$ can be written as the product of commutators of elements of $[H,H]$, each of which must be acting trivially on some neighborhood of the endpoints of $C(1)$ and $C(n)$.
Then these elements act trivially in some small enough branch $B$ and an application of \cref{eq:comm:simple} allows us to conclude that $[H,H] \leq N$, and so that $g_0 \in N$.

Finally, \cref{eq:g0} tells us that $g_0 = h [\tau_n, g_1] = [f_1, f_2] [\tau_n, g_1]$ for some $f_1, f_2 \in [H,H]$ (using again the fact that the commutator subgroup of $[H,H]$ is $[H,H]$ itself).
With this identity, the remaining elements of the generating set from \cref{rmk:comm:gen} all belong to $N$ because they are
\[ \tau_2^\zeta g_1^{-z} g_0 g_1^z \tau_2^\zeta = g_0^{g_1^z\tau_2^{\zeta}} = \big[ f_1^{g_1^z\tau_2^{\zeta}}, f_2^{g_1^z\tau_2^{\zeta}} \big] \big[ \tau_n^{g_1^z\tau_2^{\zeta}}, g_1^{g_1^z\tau_2^{\zeta}} \big], \]
where each of the two commutators satisfy \cref{eq:comm:simple}.
\end{proof}

\begin{remark}
\label{rmk:3:comm:simple}
As noted in the proof of \cref{thm:comm:simple}, our argument to prove that $\tau_2 \tau_i$ belongs to $N$ for $i = 3, \dots, n$ fails for $n=3$.
The reason is simply that, with only $3$ branches at each branch point, the elements $h_i = (ij)h$ and $h_j = (jk)h$ (with $\{i, j, k\} = \{1, 2, 3\}$ and $h$ as described in the proof) do not belong to a common rigid stabilizer of any branch, so one cannot apply \cref{eq:comm:simple}.
Every other step of the proof works, so it would suffice to find a different way to show that $\tau_2 \tau_3 \in N$ in order to prove that $[G_3, G_3]$ is simple.
However, if the commutator subgroup is not simple, the rest of the proof of \cref{thm:comm:simple} imply that a non-trivial proper normal subgroup $N$ of $[G_3, G_3]$ must include the elements $[\tau_3, g_1]$ and $\tau_2^\zeta g_1^{-z} g_0 g_1^z \tau_2^\zeta$ along with the subgroup $[H, H]$, and so it can be argued that $N$ must also contain $[H_A, H_A]$ for every rational arc $A$, by the normality of $N$ and the transitivity properties of the commutator subgroup.
\end{remark}

This question is left unanswered:

\begin{question}
\label{qst:3:comm:simple}
Is $[G_3, G_3]$ simple?
\end{question}

\subsection{A Finite Generating Set for the Commutator Subgroup}

As a final result about the commutator subgroup of $G_n$, we construct a finite generating set.
The proof follows the same strategy of \cite[Theorem 7.12]{Airplane} and also makes use of the fact that the commutator subgroup of a Thompson subgroup $H_A$ is the set of elements of $H_A$ that act trivially on some neighborhood of each of the two endpoints of the arc $A$.

Let $q_0 = \llbracket \overline{n} \rrbracket$ and $A_\star = [p_0,q_0]$ (which is the ``right arc'' attached to $p_0$) and denote by $q_1$ the point $\llbracket n i 1 \overline{n} \rrbracket$ (which is the ``middle point'' of $A_\star$).
We start from the generating set from \cref{rmk:comm:gen}, and we will show that those elements all belong to the subgroup
\[ G_\star = \langle \tau_2 \tau_3, \dots, \tau_2 \tau_n, [\tau_n, g_1], g_0, g_0^{\tau_2}, g_\star \rangle, \]
where $g_\star$ is the element of $[H,H]$ that acts on the arc $[p_0,q_1]$ (the left half of the right arc $A_\star = [p_0, q]$) as $g_0$ does on the entire arc $A_0$.
By definition $G_\star$ contains $\tau_2 \tau_i$ for each $i = 1, \dots, n$ and $[\tau_n, g_1]$, so it suffices to show that the elements $\tau_2^\zeta g_1^{-z} g_0 g_1^z \tau_2^\zeta$ all belong to $G_\star$.

For $\zeta = 0$, observe that every element $g_1^{-z} g_0 g_1^z$ has the shape sketched in \cref{fig:comm:fg}, from which we deduce that it belongs to $[H,H] g_0$.
Now, we can prove that $[H,H] \leq G_\star$ with the following strategy:
if $h \in [H,H]$, then $h^{g_0^k}$ belongs to the commutator subgroup of $H_{A_\star}$ for a negative integer $k$ that is small enough.
But $[H_{A_\star}, H_{A_\star}]$ is included in the commutator subgroup of $\langle [\tau_n, g_1], g_\star \rangle$, as the two elements act on $A_\star$ as the two generators of $F$ do on $[0,1]$ (this step is discussed in more detail in \cite[Theorem 7.12]{Airplane}).
Since $[\tau_n,g_1], g_\star \in G_\star$, we have that $[H_{A_\star}, H_{A_\star}]$ is included in $G_\star$.
Thus, $[H,H] \leq G_\star$, so $g_1^{-z} g_0 g_1^z \in G_\star$ for all $z \in \mathbb{Z}$.

As for the case $\zeta = 1$, consider the arc $A_0' = [t,q]$, where $t = \llbracket 2 \overline{n} \rrbracket$ and, as before, $q = \llbracket ni1\overline{n} \rrbracket$.
One can prove that the elements $\tau_2 g_1^{-z} g_0 g_1^z \tau_2$ (similar to \cref{fig:comm:fg}, but replacing the left central branch with the top one) all belong to $G_\star$ essentially by replacing $g_0$ with $g_0^{\tau_2}$ (that acts on $A_0'$ as $g_0$ does on $A_0$) and $H$ with $H_{A_0'}$ in the previous paragraph.
This simply means switching the first and second branches at $p_0$, which preserves the arc $A_\star$ that we used as a pivot, so one ``pushes'' the action of $[H_{A_0'}, H_{A_0'}]$ inside $A_0$ by conjugating by negative powers of $g_0^{\tau_2}$ and obtains that $[H_{A_0'}, H_{A_0'}] \leq G_\star$.

Ultimately, this shows that $G_\star \geq [G_n, G_n]$.
The converse is clear, so we ultimately have the following.

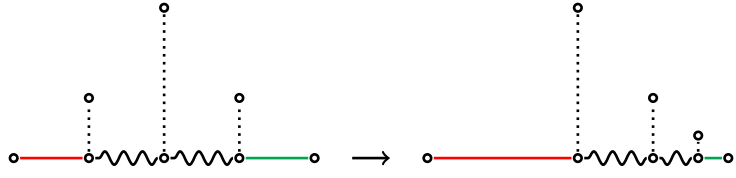
\begin{figure}
\centering
\centering
\begin{tikzpicture}
    \node[vertex] (c) at (0,0) {};
    \node[vertex] (l) at (-2,0) {};
        \node[vertex] (lc) at (-1,0) {};
        \node[vertex] (lt) at (-1,.8) {};
    \node[vertex] (t) at (0,2) {};
    \node[vertex] (r) at (2,0) {};
        \node[vertex] (rc) at (1,0) {};
        \node[vertex] (rt) at (1,.8) {};
    \draw[curvy] (rc) to (c);
    \draw[curvy] (c) to (lc);
    \draw[red] (lc) to (l);
    \draw[Green] (rc) to (r);
    \draw[dotted] (lc) to (lt);
    \draw[dotted] (rc) to (rt);
    \draw[dotted] (c) to (t);
    \draw[->] (2.5,0) to (3,0);
    \begin{scope}[xshift=5.5cm]
    \node[vertex] (c) at (0,0) {};
    \node[vertex] (l) at (-2,0) {};
    \node[vertex] (t) at (0,2) {};
    \node[vertex] (r) at (2,0) {};
            \node[vertex] (rrc) at (1.6,0) {};
            \node[vertex] (rrt) at (1.6,.3) {};
        \node[vertex] (rc) at (1,0) {};
        \node[vertex] (rt) at (1,.8) {};
    \draw[curvy] (rrc) to (rc);
    \draw[curvy] (rc) to (c);
    \draw[red] (c) to (l);
    \draw[Green] (rrc) to (r);
    \draw[dotted] (c) to (t);
    \draw[dotted] (rc) to (rt);
    \draw[dotted] (rrc) to (rrt);
    \end{scope}
\end{tikzpicture}
\caption{The general shape of the elements $g_1^{-z} g_0 g_1^z$, with an action on the curvy lines that depends on $z$.}
\label{fig:comm:fg}
\end{figure}

\begin{theorem}
\label{thm:comm:fg}
The commutator subgroup of a dendrite rearrangement group $G_n$ is finitely generated, for every $n \geq 3$.
\end{theorem}


\section{Further Properties of Dendrite Rearrangement Groups}
\label{sec:properties}

This Section is a collection of additional results about dendrite rearrangement groups.
Here we provide information about their finite subgroups (\cref{sub:fin:subgps}), we prove that they are pairwise non-isomorphic (\cref{sub:distinct}), we show how they relate to the other Thompson groups (\cref{sub:relations:Thompson}), we see that their conjugacy problem is solvable (\cref{sub:conj}), and finally we prove that they are not invariably generated (\cref{sub:IG}).

\subsection{Finite Subgroups}
\label{sub:fin:subgps}

Generalising the ideas used in \cite[Proposition 2.10]{belk2016rearrangement} (which is about the Vicsek rearrangement group, see \cref{sub:Vic}), we obtain the following result.

\begin{proposition}
\label{prop:finite:subgroups}
Each finite subgroup of $G_n$ is a subgroup of the group of automorphisms of some finite $n$-regular tree, and all such subgroups for all finite $n$-regular trees are achieved.

In particular, each finite subgroup of $G_n$ has order
\[ p_0^{\varepsilon} p_1^{k_1} \dots p_m^{k_m}, \]
where $p_1, \dots, p_m$ are the prime divisors of $(n-1)!$, $k_i \in \mathbb{N}$ (possibly zero), $\varepsilon \in \{0,1\}$ and $p_0 = n$ if it is prime and it is $1$ otherwise.
Every such order is achieved by some abelian subgroup.
\end{proposition}

\begin{proof}
First, let $G$ be a finite subgroup of $G_n$.
It is shown in \cite[Theorem 2.9]{belk2016rearrangement} that every finite subgroup of a rearrangement group is the subgroup of the automorphism group of some graph expansion of the replacement system, so $G \leq \mathrm{Aut}(T)$ for some finite $n$-regular tree $T$.
It is easy to see that every finite $n$-regular tree can be realized as a graph expansion of $\mathcal{D}_n$.

The decomposition of $\mathrm{Aut}(T)$ described in this paragraph is well-known in literature and is often called \textit{wreath recursion} (see for example \cite[Section 1.4]{NekSSG}).
It is known that every finite tree has a unique \textit{center},
which is either a vertex or an edge that is fixed by every automorphism (see for example \cite[Theorem 3.46]{GroupsGraphsTrees}).
When an edge $e$ is fixed, it is still possible that $\iota(e)$ and $\tau(e)$ are switched.
In any case, $\mathrm{Aut}(T)$ can be decomposed as a finite sequence of semidirect products (where the action is by permutations of the components of the direct product) in the following way:
\[ \mathrm{Aut}(T) = A_0 \ltimes (A_1 \times \dots \times A_n), \]
where $A_0 \leq \mathrm{Sym}(n)$ or $A_0 \leq \mathrm{Sym}(2)$ (depending on whether the center is a vertex or an edge) and each $A_i$ is
\[ A_i = A_{i,0} \ltimes (A_{i,1} \times \dots \times A_{i,n-1}), \]
where $A_{i,0} \leq \mathrm{Sym}(n-1)$
and each $A_{i,j}$ is built in this very same way, eventually ending with the trivial group.
In particular, this implies that the order of $\mathrm{Aut}(T)$ (and of each of its subgroups) is of the type exhibited in the statement.

Now, given a number of the form exhibited in the statement, we can write it as a product $q_1 \dots q_k$ of primes $q_i$'s that need not be pairwise distinct.
Without loss of generality, we order them so that $q_1 = n = p_0$ if it is prime and it is one of the factors.
Consider a rooted tree $\mathfrak{T}$ whose root has $q_1$ children whose $i$-th level vertices all have $q_{i+1}$ children (note in particular that each vertex has degree that is at most $n$, by the choice $q_1 = n$ if it is a prime factor).
Denote the root by the empty word and each $i$-th child of $w$ by $wi$.
The abelian group $C_{q_1} \times \dots \times C_{q_k}$ acts faithfully on this tree with the diagonal action defined as follows:
an element $(z_1, \dots, z_k)$ maps each vertex $x_1 \dots x_k$ to $(x_1 + q_1) \dots (x_k + q_k)$.
Since each vertex of $\mathfrak{T}$ has degree that is at most $n$, this action can be extended to a faithful action on the natural completion of $\mathfrak{T}$ to an $n$-regular tree.
In particular, this group embeds into the automorphism group of an $n$-regular tree, so we are done.
\end{proof}

\subsection{Pairwise Distinction}
\label{sub:distinct}

Using \cref{prop:finite:subgroups} from the previous Subsection, one can immediately distinguish two dendrite rearrangement groups $G_n$ and $G_m$ using the orders of their finite subgroups whenever $n$ or $m$ is at most $4$ (they can actually be distinguished in this way for all $n, m \leq 9$).
For $n \geq 5$, we can use the same argument that is mentioned in Subsection 2.3 of \cite{belk2016rearrangement} for distinguishing the Vicsek rearrangement groups.
For the sake of completeness, here we explain this argument in full detail for dendrite rearrangement groups.

\begin{proposition}
\label{prop:alt}
For every $n \geq 3$, the dendrite rearrangement group $G_n$ contains isomorphic copies of the alternating group $A_n$ but not of $A_{n+1}$.
\end{proposition}

\begin{proof}
Thanks to \cref{prop:finite:subgroups} we know what the finite subgroups of $G_n$ are, from which it is immediate to find copies of $A_n$.
More explicitly, it is clear that $A_n$ acts faithfully on the base graph of $\mathcal{D}_n$ itself as a subgroup of the permutation subgroup $K$ from \cref{sub:perm}.

If we suppose that $A_{n+1}$ embeds into $G_n$, then by \cref{prop:finite:subgroups} we must have that $A_{n+1}$ acts faithfully on a finite $n$-regular tree.
This is a contradiction because of the following:

\begin{claim}
For $m \geq 5$, if the alternating group $A_m$ acts faithfully on a tree $\mathcal{T}$ then there exists a vertex of $\mathcal{T}$ whose degree is at least $m$.
\end{claim}

To prove this Claim, assume that $A_m \simeq H \leq \mathrm{Aut}(\mathcal{T})$.
As noted in the previous \cref{sub:fin:subgps}, $H$ must fix a vertex or an edge of $\mathcal{T}$ (see for example \cite[Theorem 3.46]{GroupsGraphsTrees}).
Up to a barycentric subdivision of the fixed edge, we can assume that $H$ is fixing a vertex $r$, so we can think of $\mathcal{T}$ as if it were a rooted tree.
Then observe that, for each $k \in \mathbb{N}$, the subgroup consisting of the elements of $H$ that act trivially on the first $k$ levels of $\mathcal{T}$ is normal in $H$.
Since $H$ is simple, this implies that for every $k$ either $H$ acts trivially or faithfully on the first $k$ levels.
Thus, since $H$ is not trivial, there exists a $k$ such that the action of $H$ on the first $k-1$ levels is trivial whereas the action on level $k$ is faithful.
Note that then $H$ must embed into a direct product of symmetric groups, i.e.,
\[ H \leq \mathrm{Sym}(m_1) \times \dots \times \mathrm{Sym}(m_l), \]
where the $m_i$'s are the number of children of the $l$ vertices at level $k$.
We now prove that this cannot happen unless there exists a $j$ such that $m_j \geq m$.
This will conclude the proof, since each $m_i$ is either the degree of a vertex of $\mathcal{T}$ or the degree minus $1$, depending on whether $k=0$ or $k>0$.

Denote by $\mathrm{pr}_i$ the $i$-th projection in the aforementioned direct product.
Since $H$ is simple, each $\mathrm{pr}_i(H)$ is either trivial or isomorphic to $H \simeq A_m$.
These cannot all be trivial or $H$ would be trivial itself, so there is a $j$ such that $A_m \simeq \mathrm{pr}_j(H) \leq \mathrm{Sym}(m_j)$.
It is known that $\mathrm{Sym}(m_j)$ cannot contain a copy of $A_m$ unless $m \leq m_j$, so we are done.
\end{proof}

As an immediate consequence of \cref{prop:alt}, we have the following.

\begin{theorem}
\label{thm:non:isomorphic}
$G_n \simeq G_m$ if and only if $n = m$.
\end{theorem}

\begin{remark}
\label{rmk:rubin}
Very often results of this kind can be proved as applications of theorems of reconstructions of topological spaces from groups acting on them (i.e., given a group $G$ acting ``nicely'' on two ``nice'' topological spaces $X$ and $Y$, one finds an equivariant homeomorphism between $X$ and $Y$).
The most prominent example of this is the celebrated paper \cite{Rubin}, which features 4 families of group actions that allow a reconstruction of the underlying space.
Of these, the result named \textit{Theorem 2} in the abstract of \cite{Rubin}, which is often just called \textbf{Rubin's Theorem} and is also discussed and proved more succinctly in \cite{ShortRubin}, applies to many Thompson groups and other ``rich'' subgroups of homeomorphism groups.

However, we remark here that neither Rubin's Theorem (nor any of the other 3 from \cite{Rubin}, to the author's knowledge) applies to $G_n$, nor to any group acting by homeomorphisms on a Wa\.zewski dendrite $D_n$.
Indeed, let $U$ be the open set $C^\mathrm{o}(11) \cup \dots \cup C^\mathrm{o}(n1) \cup \{p_0\}$ (which is essentially obtained by taking the inner ``halves'' of the $n$ branches at $p_0$).
Every element of $\mathrm{Rist}_{G_n}(U)$ must fix setwise each of the cells $C(11), \dots, C(n1)$, so it must also fix $p_0$, which prevents $\mathrm{Rist}_{G_n}(U)$ from satisfying the conditions of Rubin's Theorem.
\end{remark}

Before moving on to a comparison with other Thompson groups, we highlight the following fact.

\begin{proposition}
\label{prop:inclusion}
$G_n < G_{n+1}$, for all $n \geq 3$.
\end{proposition}

\begin{proof}
Simply consider the subgroup of $G_{n+1}$ generated by its Thompson subgroup $H$ along with the subgroup of $K$ that fixes the cell $C(n+1)$, which is the subgroup of those rearrangements that ``rigidly drag'' every $(n+1)$-th branch along with its branch point.
More precisely, each cell $C(x)$ corresponding to a word $x$ that ends with $n+1$ and does not feature $n+1$ anywhere else is mapped canonically to a cell $C(y)$ with this same property.
This is an isomorphic copy of $G_n$ inside $G_{n+1}$.
\end{proof}

\subsection{Relations with other Thompson Groups}
\label{sub:relations:Thompson}

As seen in \cref{sub:thomp}, dendrite rearrangement groups feature infinitely many isomorphic copies of Thompson's group $F$.
Here we show how the $G_n$'s compare to the bigger siblings of $F$.

\begin{proposition}
\label{prop:dendrite:Thompson:comparison}
For every $n \geq 3$, each dendrite rearrangement group $G_n$ embeds in Thompson's group $V$, but $G_n$ does not embed in Thompson's group $T$ nor does $T$ embed into $G_n$ (so, in particular, $V$ also does not embed into $G_n$).
\end{proposition}

\begin{proof}
The embedding of $G_n$ into $V$ is also described, more broadly, in Remark 1.8 from \cite{RearrConj}, and it essentially consists of ``forgetting'' about edge adjacency and thus allowing every permutation of the edges.
To achieve this, one can simply consider a graph pair diagram for $G_n$ and ``unglue'' all of the edges, considering domain and range graphs consisting of edges each of which is not adjacent to any other.
The result is a valid graph pair diagram for the Higman-Thompson group $V_n$, since every permutation of the edges is allowed in $V_n$.
This prompts a natural embedding of $G_n$ inside $V_n$, which in turn embeds in $V$ by \cite[Theorem 7.2]{higman1974finitely}.

As for Thompson's group $T$, it is known that its finite subgroups are cyclic and that every finite cyclic group embeds in $T$.
Then \cref{prop:finite:subgroups} implies that $G_n$ cannot embed into $T$ (as clearly, for every $n \geq 3$, there are automorphism groups of finite $n$-regular trees that are not cyclic) and also that $T$ cannot embed into $G_n$ (otherwise $G_n$ would have to feature cyclic subgroups of every order, but it does not by the aforementioned Proposition).
\end{proof}

Since both the rearrangement groups of the Basilica and of the Airplane limit spaces (respectively $T_B$ from \cite{Belk_2015} and $T_A$ from \cite{belk2016rearrangement} and \cite{Airplane}) contain copies of $T$ and embed in $T$, we can also conclude the following.

\begin{corollary}
For every $n \geq 3$, $T_B$ and $T_A$ do not embed into $G_n$ nor does $G_n$ into $T_B$ or $T_A$. 
\end{corollary}

It appears that the groups $T_B$ and $T_A$, along with Thompson's group $T$, have more in common with the orientation-preserving dendrite rearrangement groups mentioned in \cref{sub:orientation} below.
Instead, dendrite rearrangement groups $G_n$ do not preserve an orientation of the limit space $D_n$ on which they act, which is arguably what separates these groups from $T_B$ and $T_A$.

The only other rearrangement group mentioned in literature that is likely to contain a copy of or embed into some dendrite rearrangement group is the Vicsek rearrangement group, along with its generalizations (see \cite[Example 2.1]{belk2016rearrangement}).
See \cref{sub:Vic} below for more about this comparison.

\subsection{A Solution to the Conjugacy Problem}
\label{sub:conj}

We say that the \textbf{conjugacy problem} of a group $G$ is solvable if there exists an algorithm that, given two elements of $G$, infallibly decides in finite time whether those elements belong to the same conjugacy class.

The main result of \cite{RearrConj} states that, given an expanding replacement system whose replacement rules
are reduction-confluent, the conjugacy problem
is solvable in the associated rearrangement group.
The solution described in \cite{RearrConj} was inspired by a solution to the conjugacy problem for Thompson groups $F$, $T$ and $V$ from \cite{Belk2007ConjugacyAD} and makes use of \textit{strand diagrams}.
The following paragraphs succinctly explain what reduction-confluent means and why dendrite replacement systems satisfy such property.

A \textit{graph reductions} is the opposite operation of graph expansions of the replacement system.
More precisely, given a graph $\Gamma$, let $\Lambda$ be a subgraph isomorphic to a replacement graph $R_c$.
Let $\phi$ be such an isomorphism and suppose that the image of every vertex of $R_c$ other than its initial and terminal vertices is not incident on vertices of $\Gamma \setminus \phi(R_c)$ (i.e., only the initial and terminal vertices are allowed to be interact with the rest of $\Gamma$).
Then we replace $\phi(R_c)$ with an edge colored by $c$ with initial vertex $\phi(\iota)$ and terminal vertex $\phi(\tau)$.
This defines a graph rewriting system, whose objects are finite graphs (up to isomorphisms) and whose operation of reduction we denote by $\longrightarrow$.
When edges are undirected (\cref{rmk:undirected}), we do not distinguish between their orientations for the purposes of determining graph isomorphisms and graph reductions.

A rewriting system is \textbf{confluent} if, whenever $A \dashrightarrow B_1$ and $A \dashrightarrow B_2$ are finite sequences of rewritings, there exist an object $C$ and a finite (possibly empty) sequences of rewritings $B_1 \dashrightarrow C$ and $B_2 \dashrightarrow C$ (i.e., when rewriting the same object $A$ in two different ways, one can always continue rewriting to a common object $C$).
We say that a replacement system is \textbf{reduction-confluent} if the the rewriting system of its graph reductions is confluent.

Let us see why dendrite replacement systems $\mathcal{D}_n$ have confluent graph reductions.
Since graph reductions decrease the number of edges, they are terminating (i.e., there is no infinite sequence of graph reductions).
Then, by Newman's diamond Lemma \cite{NewmanDiamond}, it suffices to show that they are \textbf{locally confluent}, i.e., whenever a graph $\Gamma$ and two distinct reductions $\Gamma \longrightarrow \Gamma_1$ and $\Gamma \longrightarrow \Gamma_2$, there exist a graph $\Gamma^*$ and a finite (possibly empty) sequences of reductions $\Gamma_1 \dashrightarrow \Gamma^*$ and $\Gamma_2 \dashrightarrow \Gamma^*$.
In the case of dendrites, a graph reduction is uniquely identified by a vertex of degree $n$ and its $n$ incident edges.
Assume first that the two reductions $\Gamma \longrightarrow \Gamma_1$ and $\Gamma \longrightarrow \Gamma_2$ involve no common edges, such as those portrayed in \cref{fig:parallel:reductions} (edges are drawn without orientation because they are undirected, see \cref{rmk:undirected}).
Then such reductions can be performed one after the other, whatever the order, meaning that $\Gamma \longrightarrow \Gamma_1 \longrightarrow \Gamma^*$ and $\Gamma \longrightarrow \Gamma_2 \longrightarrow \Gamma^*$, as needed.
Assume instead that the graphs involved in the two reductions $\Gamma \longrightarrow \Gamma_1$ and $\Gamma \longrightarrow \Gamma_2$ share some edge.
Then they can only share one edge, otherwise they would be the same reduction.
The subgraph involved in both reductions is thus the one represented in \cref{fig:reductions}, from which it is clear that $\Gamma_1$ and $\Gamma_2$ are isomorphic, so in this case we have $\Gamma \longrightarrow \Gamma_1 \simeq \Gamma^*$ and $\Gamma \longrightarrow \Gamma_2 \simeq\Gamma^*$.
Ultimately, dendrite replacement systems are reduction-confluent, so their conjugacy problem is solvable by \cite{RearrConj}.

\begin{figure}
\centering
\begin{tikzpicture}[scale=.9]
    \node at (0,2.2) {$\Gamma$};
    \node[vertex] (C) at (0,0) {};
    \node[vertex] (L) at (-1.5,0) {};
    \node[vertex] (T) at (0,1.5) {};
    \node[vertex] (R) at (1.5,0) {};
    \node[vertex] (B) at (0,-1.5) {};
    \node[vertex] (LL) at (-3,0) {};
    \node[vertex] (LT) at (-1.5,1.5) {};
    \node[vertex] (LB) at (-1.5,-1.5) {};
    \node[vertex] (RR) at (3,0) {};
    \node[vertex] (RT) at (1.5,1.5) {};
    \node[vertex] (RB) at (1.5,-1.5) {};
    \draw (C) to (L);
    \draw (C) to (T);
    \draw (C) to (R);
    \draw (C) to (B);
    \draw (L) to (LL);
    \draw (L) to (LT);
    \draw (L) to (LB);
    \draw (R) to (RR);
    \draw (R) to (RT);
    \draw (R) to (RB);
    \draw[dotted,magenta] (L) circle (1.65);
    \draw[dashed,magenta,-to] (-3.5,0) to[out=180,in=90] (-5.25,-4);
    \draw[dotted,green] (R) circle (1.65);
    \draw[dashed,green,-to] (3.5,0) to[out=0,in=90] (5.25,-4);
    \begin{scope}[xshift=-3.75cm,yshift=-4.75cm]
    \node at (0,2.2) {$\Gamma_1$};
    \node[vertex] (LC) at (-.75,0) {};
    \node[vertex] (LL) at (-2.25,0) {};
    \node[vertex] (LT) at (-.75,1.5) {};
    \node[vertex] (LB) at (-.75,-1.5) {};
    \node[vertex] (RC) at (.75,0) {};
    \node[vertex] (RR) at (2.25,0) {};
    \node[vertex] (RT) at (.75,1.5) {};
    \node[vertex] (RB) at (.75,-1.5) {};
    \draw[magenta] (LC) to (LL);
    \draw (LC) to (LT);
    \draw (LC) to (LB);
    \draw (LC) to (RC);
    \draw (RC) to (RR);
    \draw (RC) to (RT);
    \draw (RC) to (RB);
    \end{scope}
    \begin{scope}[xshift=3.75cm,yshift=-4.75cm]
    \node at (0,2.2) {$\Gamma_2$};
    \node[vertex] (LC) at (-.75,0) {};
    \node[vertex] (LL) at (-2.25,0) {};
    \node[vertex] (LT) at (-.75,1.5) {};
    \node[vertex] (LB) at (-.75,-1.5) {};
    \node[vertex] (RC) at (.75,0) {};
    \node[vertex] (RR) at (2.25,0) {};
    \node[vertex] (RT) at (.75,1.5) {};
    \node[vertex] (RB) at (.75,-1.5) {};
    \draw (LC) to (LL);
    \draw (LC) to (LT);
    \draw (LC) to (LB);
    \draw (LC) to (RC);
    \draw[green] (RC) to (RR);
    \draw (RC) to (RT);
    \draw (RC) to (RB);
    \end{scope}
\end{tikzpicture}
\caption{Two reductions of graphs that share no edge, for the replacement system $\mathcal{D}_4$.}
\label{fig:parallel:reductions}
\end{figure}

\begin{figure}
\centering
\begin{tikzpicture}[scale=.9]
    \node at (0,2.2) {$\Gamma$};
    \node[vertex] (LC) at (-.75,0) {};
    \node[vertex] (LL) at (-2.25,0) {};
    \node[vertex] (LT) at (-.75,1.5) {};
    \node[vertex] (LB) at (-.75,-1.5) {};
    \node[vertex] (RC) at (.75,0) {};
    \node[vertex] (RR) at (2.25,0) {};
    \node[vertex] (RT) at (.75,1.5) {};
    \node[vertex] (RB) at (.75,-1.5) {};
    \draw (LC) to (LL);
    \draw (LC) to (LT);
    \draw (LC) to (LB);
    \draw (LC) to (RC);
    \draw (RC) to (RR);
    \draw (RC) to (RT);
    \draw (RC) to (RB);
    \draw[dotted,magenta] (LC) circle (1.65);
    \draw[dashed,magenta,-to] (-2.75,0) to[out=180,in=90] (-4.25,-3.5);
    \draw[dotted,green] (RC) circle (1.65);
    \draw[dashed,green,-to] (2.75,0) to[out=0,in=90] (4.25,-3.5);
    \begin{scope}[xshift=-3.25cm,yshift=-4.75cm]
    \node at (0,2.2) {$\Gamma_1$};
    \node[vertex] (C) at (0,0) {};
    \node[vertex] (L) at (-1.5,0) {};
    \node[vertex] (T) at (0,1.5) {};
    \node[vertex] (R) at (1.5,0) {};
    \node[vertex] (B) at (0,-1.5) {};
    \draw[magenta] (C) to (L);
    \draw (C) to (T);
    \draw (C) to (R);
    \draw (C) to (B);
    \draw[dotted,magenta] (-.75,0) circle (.9);
    \end{scope}
    \node at (0,-4.75) {\LARGE$\simeq$};
    \begin{scope}[xshift=3.25cm,yshift=-4.75cm]
    \node at (0,2.2) {$\Gamma_2$};
    \node[vertex] (C) at (0,0) {};
    \node[vertex] (L) at (-1.5,0) {};
    \node[vertex] (T) at (0,1.5) {};
    \node[vertex] (R) at (1.5,0) {};
    \node[vertex] (B) at (0,-1.5) {};
    \draw (C) to (L);
    \draw (C) to (T);
    \draw[green] (C) to (R);
    \draw (C) to (B);
    \draw[dotted,green] (.75,0) circle (.9);
    \end{scope}
\end{tikzpicture}
\caption{Two reductions of graphs that share an edge, for the replacement system $\mathcal{D}_4$.}
\label{fig:reductions}
\end{figure}

\begin{proposition}
Dendrite rearrangement groups have solvable conjugacy problem.
\end{proposition}

We note here that conjugacy in the (topological) group $\mathrm{AAut}(\mathcal{T}_{d,k})$ of almost automorphism of trees has been studied recently in \cite{Goffer2019ConjugacyAD}.
That paper shows that whether two hyperbolic elements of $\mathrm{AAut}(\mathcal{T}_{d,k})$ are conjugate can be established by studying elements of the Higman-Thompson group $V_{d,k}$ that ``approximate'' the two almost automorphisms.
Since we have seen that dendrite rearrangement groups ``approximate'' the full groups of homeomorphisms (\cref{thm:dense}), inspired by the use that \cite{Goffer2019ConjugacyAD} made of the strand diagrams developed in \cite{Belk2007ConjugacyAD}, it is natural to ask the following.

\begin{question}
Can the technology of strand diagrams for rearrangement groups developed in \cite{RearrConj} be used to study conjugacy in $\mathrm{Homeo}(D_n)$?
\end{question}

\subsection{Absence of Invariable Generation}
\label{sub:IG}

We say that a group $G$ is \textbf{invariably generated} if it admits a generating set $S$ that still generates $G$ even after one has conjugated each of the elements of $S$ by elements of $G$.
More precisely:
\[ G = \langle s^{g_s} \mid s \in S \rangle \text{ for every choice of } \{g_s \mid s \in S \} \subseteq G. \]
It was proved in \cite{Gelander2016InvariableGO} that Thompson's group $F$ is invariably generated, whereas the bigger siblings $T$ and $V$ are not.

A rearrangement group is \textbf{weakly cell-transitive} if, given an arbitrary cell $C$ and an arbitrary proper union $K$ of finitely many cells, there exists a rearrangement that maps $K$ inside $C$.
It was shown in \cite[Proposition 3.3]{IG} that this property of a rearrangement group is equivalent to other transitivity properties that can be found in literature, namely \textit{CO-transitivty} and \textit{flexibility}.

The main result of \cite{IG} is that, if a rearrangement group is weakly cell-transitive, then it is not invariably generated, nor is its commutator subgroup.

\begin{proposition}
\label{prop:wct}
Dendrite rearrangement groups are weakly cell-transitive.
\end{proposition}

\begin{proof}
Given a cell $C$ of $D_n$ and a union $K$ of finitely many cells we need to show that $G_n$ can map $K$ inside $C$.
Note that every cell must contain a branch (in fact, it contains infinitely many) and that the union of any finite amount of cells is fully included in some branch.
Then it suffices to know that, given two arbitrary branches $B_1$ and $B_2$, there exists an element $g$ of $G_n$ such that $g(B_1) \subseteq B_2$.
But then, thanks to \cref{cor:comm:trans}, we are already done.
\end{proof}

This prompts the following application of the main result from \cite{IG}, and gives an example of behaviour of $G_n$ that makes it appear more similar to Thompson groups $T$ and $V$ than to $F$ (see \cref{sub:amen} below for another example of this).

\begin{proposition}
\label{prop:IG}
Dendrite rearrangement groups and their commutator subgroups are not invariably generated.
\end{proposition}

\subsection{Existence of Free Subgroups}
\label{sub:amen}

As the result in the previous Subsection, the following is another example of a behaviour of dendrite rearrangement groups that resemble that of Thompson groups $T$ and $V$ rather than $F$.

\begin{proposition}
Dendrite rearrangement groups contain non-abelian free groups.
In particular, they are not amenable.
\end{proposition}



This can be seen as a consequence of weak cell-transitivity (\cref{prop:wct}), since by \cite[Corollary 4.3]{IG} weakly cell-transitive rearrangement groups include non-abelian free groups.
It can also be seen as a consequence of \cite{Shi12}, whose main result states that a group acting minimally on a non-degenerate dendrite includes non-abelian free groups.
Minimality for dendrite rearrangement groups follows from \cite[Proposition 3.3]{IG} or can be proved using \cite[Theorem 5.2]{GM19}.

As a third option, here is a direct proof using the classical ping-pong argument.
We need to find $g,h \in G_n$ and disjoint subsets $X,Y$ of $D_n$ such that $g^z(X) \subseteq Y$ and $h^z(Y) \subseteq X$, for all $z \in \mathbb{Z}\setminus\{0\}$.
Let $g \coloneq g_0^2$ and let $h$ be the element depicted in \cref{fig:free:subgroup:generator} (the figure works for $n=3$; for higher $n$, one may consider any extension to the missing cells, which are those corresponding to words that include digits between $3$ and $n-1$).
Let $X = C^\mathrm{o}(1n) \cup C^\mathrm{o}(n)$ and $Y = C^\mathrm{o}(12) \cup C^\mathrm{o}(2)$, which are disjoint.
Now it is easy to check that $g(Y) \subseteq C^\mathrm{o}(1n)$ and that $g\left(C^\mathrm{o}(1n)\right) \subseteq C^\mathrm{o}(1n)$, so $g^z(Y) \subseteq X$ for all $z>0$.
Analogous arguments show that $g^z(Y) \subseteq X$ for $z<0$ and that $h^z(X) \subseteq Y$ for all $z>0$ and $z<0$.

\begin{figure}
\centering
\begin{tikzpicture}[scale=1.3333]
    \node[vertex] (LL) at (0,0) {};
    \node[vertex] (L) at (1,0) {};
        \node[vertex] (LT) at ($(L)+(90:.75)$) {};
        \node[vertex] (LC) at ($(L)+(90:.375)$) {};
        \node[vertex] (LR) at ($(LC)+(0:.375)$) {};
    \node[vertex] (C) at (2,0) {};
        \node[vertex] (CT) at ($(C)+(90:1.5)$) {};
    \node[vertex] (R) at (4,0) {};
    \draw (LL) to node[below]{\scriptsize$1n$} (L);
    \draw (L) to node[below]{\scriptsize$11$} (C);
        \draw (LC) to node[left]{\scriptsize$121$} (L);
        \draw (LC) to node[left]{\scriptsize$12n$} (LT);
        \draw (LC) to node[xshift=.06667cm,yshift=.2cm]{\scriptsize$122$} (LR);
    \draw (C) to node[below]{\scriptsize$n$} (R);
        \draw (C) to node[right]{\scriptsize$2$} (CT);
    \draw[->] (4.3333,0) to node[above]{$h$} (4.6667,0);
    \begin{scope}[xshift=5cm]
    \node[vertex] (LL) at (0,0) {};
    \node[vertex] (C) at (2,0) {};
        \node[vertex] (CT) at ($(C)+(90:1.5)$) {};
        \node[vertex] (CC) at ($(C)+(90:.75)$) {};
            \node[vertex] (CTC) at ($(CC)+(90:.375)$) {};
            \node[vertex] (CTR) at ($(CTC)+(0:.375)$) {};
        \node[vertex] (CR) at ($(CC)+(0:.75)$) {};
    \node[vertex] (R) at (4,0) {};
    \draw (C) to node[below]{\scriptsize$h(12n)=1$} (LL);
    \draw (C) to node[below]{\scriptsize$h(122)=n$} (R);
    \draw (CC) to node[left]{\scriptsize$h(121)=21$} (C);
    \draw (CC) to node[xshift=.3cm,yshift=-.2cm]{\scriptsize$h(1n)=22$} (CR);
    \draw (CC) to node[left]{\scriptsize$h(11)=2n1$} (CTC);
    \draw (CTC) to node[xshift=.57cm,yshift=-.2cm]{\scriptsize$h(n)=2n2$} (CTR);
    \draw (CTC) to node[left]{\scriptsize$h(2)=2nn$} (CT);
    \end{scope}
\end{tikzpicture}
\caption{An element $h$ that plays ping-pong with $g_0^2$.}
\label{fig:free:subgroup:generator}
\end{figure}


\section{Other Dendrite Rearrangement Groups}
\label{sec:generalizations}

The main result of this Section is the density of the Airplane rearrangement group $T_A$ in the orientation-preserving subgroup $\mathbb{H}_\infty^+$ of the full homeomorphism group of the infinite-order Wa\.zewski dendrite $D_\infty$ (\cref{thm:dns:A}).
This will be proved with an application of the same strategy used in \cref{sec:dns}, also employing results from \cite{Airplane}.

The rest of the Section is devoted to introducing multiple possible generalizations of dendrite rearrangement groups, of which $T_A$ is one example.
We will first briefly discuss the Vicsek rearrangement groups from \cite{belk2016rearrangement} (\cref{sub:Vic}), then we will introduce an orientation-preserving version of dendrite rearrangement groups (\cref{sub:orientation}) which will prompt the aforementioned result about the Airplane rearrangement group $T_A$ (\cref{sub:Air}), and finally we will exhibit replacement systems for the generalized Wa\.zewski dendrites $D_S$ (\cref{sub:gen:den}).
With the exception of the results about $T_A$, this Section is mostly speculative.

\subsection{The Vicsek rearrangement groups}
\label{sub:Vic}

Introduced in \cite[Example 2.1]{belk2016rearrangement}, the Vicsek replacement systems are the same as the dendrite replacement systems $\mathcal{D}_n$ except for having an additional edge originating at the initial vertex.
For example, \cref{fig:vicsek:rep:sys} depicts the Vicsek replacement system for $n=4$.

\begin{figure}
\centering
\begin{tikzpicture}
    \node at (-2.2,0) {$\Gamma =$};
    \node[vertex] (C) at (0,0) {};
    \node[vertex] (L) at (-1.5,0) {};
    \node[vertex] (T) at (0,1.5) {};
    \node[vertex] (R) at (1.5,0) {};
    \node[vertex] (B) at (0,-1.5) {};
    \draw[edge] (C) to (L);
    \draw[edge] (C) to (T);
    \draw[edge] (C) to (R);
    \draw[edge] (C) to (B);
    \begin{scope}[xshift=7.5cm]
    \node at (-3.2,0) {$R =$};
    \node[vertex] (C) at (0,0) {};
    \node[vertex] (I) at (-2.5,0) {}; \draw (-2.5,0) node[above]{$\iota$};
    \node[vertex] (L) at (-1.25,0) {};
    \node[vertex] (T) at (0,1.25) {};
    \node[vertex] (R) at (1.25,0) {}; \draw (1.25,0) node[above]{$\tau$};
    \node[vertex] (B) at (0,-1.25) {};
    \draw[edge] (C) to (L);
    \draw[edge] (C) to (T);
    \draw[edge] (C) to (R);
    \draw[edge] (C) to (B);
    \draw[edge] (I) to (L);
    \end{scope}
\end{tikzpicture}
\caption{The Vicsek replacement system for $n=4$.}
\label{fig:vicsek:rep:sys}
\end{figure}
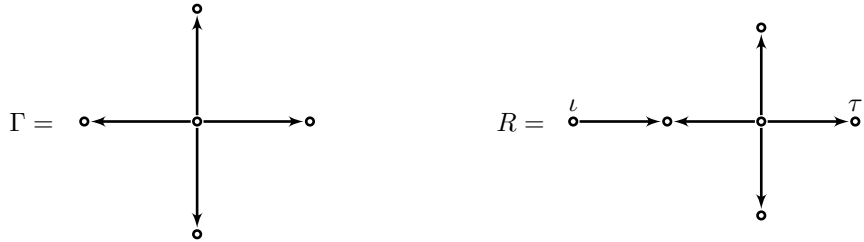

It is clear that each Vicsek replacement system has the same Wa\.zewski dendrite $D_n$ as a limit space (although dendrite replacement systems $\mathcal{D}_n$ are arguably more natural replacement systems for these limit spaces, while the Vicsek replacement system reflects some of the metric aspects of the Vicsek fractals).
It seems likely that the $n$-th Vicsek rearrangement group is generated by a copy of the Higman-Thompson group $F_3$ along with the permutation group $\mathrm{Sym}(n)$ in a similar way in which $G_n$ is generated by $F$ and $\mathrm{Sym}(n)$ by \cref{thm:gen}.
It is known that $F$ and $F_3$ are non-isomorphic groups, as they have different abelianization, which was proved in \cite{BROWN198745}.
However, it does not seem immediately clear whether or not the $n$-th Vicsek rearrangement group is isomorphic to $G_n$, since Rubin's theorem cannot be immediately applied (as discussed in \cref{rmk:rubin}) and the finite subgroups (\cref{sub:fin:subgps}) are likely the same.

Additionally, one may be inspired by this example to build replacement systems with even more edges between the initial vertex $\iota$ and the branch point of the replacement graph, which may result in a rearrangement group that has some Higman-Thompson group $F_m$ acting on each EE-arc, in place of $F$ or $F_3$.

We did not investigate these questions any further.

\subsection{Orientation-Preserving Dendrite Rearrangement Groups}
\label{sub:orientation}

The \textbf{oriented dendrite replacement system} $\boldsymbol{\mathcal{D}_n^+}$, depicted in \cref{fig:rep:sys:+}, is the 2-colored replacement system whose base graph and black replacement graph are the same graph consisting of a directed closed path of $n$ red edges and $n$ vertices, along with a black edge originating from each of these $n$ vertices, each terminating at one of $n$ other vertices.
The initial and terminal vertices $\iota$ and $\tau$ of the black replacement graph are two distinct vertices of in-degree $1$ and out-degree $0$.
The red replacement graph is a single red edge.

\begin{figure}
\centering
\begin{tikzpicture}
    \node at (-2,0) {$\Gamma =$};
    \node[vertex] (c1) at (0:.75) {};
    \node[vertex] (c2) at (72:.75) {};
    \node[vertex] (c3) at (144:.75) {};
    \node[vertex] (cn-1) at (216:.75) {};
    \node[vertex] (cn) at (288:.75) {};
    \node[vertex] (1) at (0:2) {};
    \node[vertex] (2) at (72:2) {};
    \node[vertex] (3) at (144:2) {};
    \node[vertex] (n-1) at (216:2) {};
    \node[vertex] (n) at (288:2) {};
    \draw[edge] (c1) to node[above]{$1$} (1);
    \draw[edge] (c2) to node[above left]{$2$} (2);
    \draw[edge,dotted] (c3) to (3);
    \draw[edge,dotted] (cn-1) to (n-1);
    \draw[edge] (cn) to node[right]{$n$} (n);
    \draw[edge,red] (c1) to[out=90,in=-12,looseness=.85] node[right]{$r_1$} (c2);
    \draw[edge,red,dotted] (c2) to[out=162,in=54,looseness=.85] (c3);
    \draw[edge,red,dotted] (c3) to[out=234,in=126,looseness=.85] (cn-1);
    \draw[edge,red,dotted] (cn-1) to[out=306,in=198,looseness=.85] (cn);
    \draw[edge,red] (cn) to[out=18,in=270,looseness=.85] node[right]{$r_n$} (c1);
    \begin{scope}[xshift=5cm,yshift=-1cm]
    \node at (0,2.5) {$R_{\text{black}} =$};
    \node[vertex] (b1) at (180:.75) {};
    \node[vertex] (b2) at (120:.75) {};
    \node[vertex] (bn-1) at (60:.75) {};
    \node[vertex] (bn) at (0:.75) {};
    \node[vertex] (r1) at (180:2) {}; \draw (180:2) node[above]{$\iota$};
    \node[vertex] (r2) at (120:2) {};
    \node[vertex] (rn-1) at (60:2) {};
    \node[vertex] (rn) at (0:2) {}; \draw (0:2) node[above]{$\tau$};
    \draw[edge] (b1) to node[below]{$1$} (r1);
    \draw[edge,dotted] (b2) to (r2);
    \draw[edge,dotted] (bn-1) to (rn-1);
    \draw[edge] (bn) to node[below]{$n$} (rn);
    \draw[edge,red] (bn-1) to[out=-30,in=90,looseness=1] node[right]{$r_{n-1}$} (bn);
    \draw[edge,red,dotted] (b2) to[out=30,in=150,looseness=1] (bn-1);
    \draw[edge,red] (b1) to[out=90,in=210,looseness=1] node[left]{$r_1$} (b2);
    \draw[edge,red] (bn) to[out=270,in=270,looseness=1.45] node[below]{$r_n$} (b1);
    \end{scope}
    \begin{scope}[xshift=8.75cm,yshift=-.5cm]
    \node at (0,1) {$R_{\text{\textcolor{red}{red}}} =$};
    \node[vertex] (ir) at (-.75,0) {}; \draw (-.75,0) node[above]{$\iota$};
    \node[vertex] (tr) at (.75,0) {}; \draw (.75,0) node[above]{$\tau$};
    \draw[edge,red] (ir) to node[above]{$r$} (tr);
    \end{scope}
\end{tikzpicture}
\caption{A schematic depiction of the oriented dendrite replacement system $\mathcal{D}_n^+$.}
\label{fig:rep:sys:+}
\end{figure}
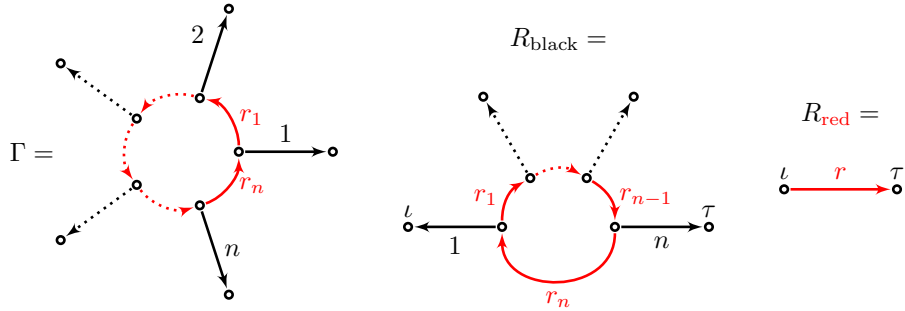

$\mathcal{D}_n^+$ is not an expanding replacement system (\cref{def:expanding}), since the red replacement graph only has two vertices linked by an edge.
Even if this does not pose an obstacle to the definition of a rearrangement group $G_n^+$ based on $\mathcal{D}_n^+$ (it would not be defined as a group of homeomorphisms, but only as a subgroup of the topological full group of a one-sided subshift of finite type), the gluing relation (\cref{def:gluing}) might not be (and in fact is not) an equivalence relation, so the limit space cannot be defined in the usual way (\cref{def:limit:space}).
However, since the gluing relation of $\mathcal{D}_n^+$ is reflexive and symmetric, it suffices to take the quotient of the symbol space of $\mathcal{D}_n^+$ by the transitive closure of the gluing relation in order to obtain an equivalence relation and a well defined limit space.

As noted in \cite[Remark 1.23]{belk2016rearrangement}, this limit space might not be as well-behaved as in the expanding case.
In our case, if we say that an edge is \textbf{null-expanding} when the replacement graph of its color consists of a sole edge, it is easy to see that each sequence $x \bar{r}$, where $xr$ is a null-expanding edge, ends up collapsing to a single point in the limit space, so the $n$ red cells $C(y r_1), \dots, C(y r_n)$ (for any finite word $y$ ending with a black edge) all degenerate to the same single point (these will correspond exactly to the branch points).
It is easy then to check that the limit spaces of $\mathcal{D}_n$ and $\mathcal{D}_n^+$ are the same, so the possible issues of the limit space not being well-behaved do not show up here.
More precisely, there is a canonical homeomorphism between the two limit spaces given by mapping each sequence $\omega$ of $\mathcal{D}_n^+$ that does not feature null-expanding edges to the same sequence of $\mathcal{D}_n$.
Every other sequences of $\mathcal{D}_n^+$ must be of the form $x r_i \bar{r}$, so it represents the same branch point as $x i 1 \bar{n}$ and thus are mapped to this same sequence of $\mathcal{D}_n$.

One can then define the rearrangement groups $G_n^+$ for the replacement systems $\mathcal{D}_n^+$ by associating to each graph pair diagram a homeomorphism in the usual way.
The addition of the red edges will simply force every rearrangement to preserve the orientation of the dendrite, by which we mean that they preserve a circular ordering of the set of branches at each branch point.
We leave the study of these orientation-preserving dendrite groups for future research, but we list here a few remarks about them:
\begin{enumerate}
    \item The group $G_3^+$ is likely isomorphic to the finitely generated Thompson-like group studied in the dissertation \cite{Dendrite}.
    That is the group of those homeomorphisms of $S^1$ that preserve a lamination induced from the Julia set for the complex map $z \to z^2 + i$ (which is homeomorphic to $D_3$).
    We suspect that every $G_n^+$ can be built as a group of homeomorphisms of the circle by using laminations in this way.
    \item The edges of $\mathcal{D}_n^+$ are not undirected, unlike those of $\mathcal{D}_n$ (see \cref{rmk:undirected}).
    \item The permutation subgroup from \cref{sub:perm} translates to a subgroup of $G_n^+$ that is cyclic of order $n$ instead of copies of the symmetric group on $n$ elements.
    \item The Thompson subgroups from \cref{lem:thomp:trans} do not seem to immediately translate to $G_n^+$, but there are probably copies of Thompson's group $F$ inside $G_n^+$ that act on certain arcs of $D_n$.
    \item The transitivity property described in \cref{prop:trans} for $G_n$ may translate to an extension of isomorphisms between trees with a rotation system.
    If this happens, it might be used to prove that $G_n^+$ is dense in the group $\mathbb{H}_n^+$ of all orientation-preserving homeomorphisms of $D_n$ (i.e., the group of those homeomorphisms that at each branch point preserve a circular order of the branches).
    \item The parity map from \cref{sub:parity} would be trivial in $G_n^+$, so the abelianization may simply be $\mathbb{Z}$ (which is what happens in the Airplane rearrangement group $T_A$, see the next Subsection for the relation between $T_A$ and dendrites).
    Note that in \cite{Dendrite} it is conjectured that the abelianization of this dendrite-Thompson group is isomorphic to $\mathbb{Z} \oplus \mathbb{Z}_2 \oplus \mathbb{Z}_3$, picturing a much more complex situation.
    \item If one chooses a different terminal vertex $\tau$ in the black replacement graph (\cref{fig:rep:sys:+}), then some of these remarks do not hold anymore (for example, if $n$ is even and $\tau$ is the terminal vertex of the $n/2$-th edge, then the black edges are actually undirected).
\end{enumerate}

Moreover, it looks like the Kaleidoscopic group $\mathcal{K}(A_3)$ studied in \cite{OrientationD3} (and their generalizations $\mathcal{K}(C_n)$) include the orientation-preserving dendrite rearrangement group $G_3^+$ (and $G_n^+$, respectively) as subgroups.
Indeed, the groups $\mathcal{K}(C_n)$ seem to consist of homeomorphisms of $D_n$ that preserve an orientation at each branch point.
It may then be possible that $G_n^+$ is dense in $\mathcal{K}(C_n)$, but we did not investigate this here.

\subsection{Infinite Order and the Airplane Rearrangement Group}
\label{sub:Air}

The author could not find a natural generalization of the dendrite replacement systems $\mathcal{D}_n$ that produces the infinite-order Wa\.zewski Dendrite $D_\infty$ (where each branch point has countably infinite order; see \cite{Duc20} for an in-depth study of the topological properties of the group $\mathrm{Homeo}(D_\infty)$).
Instead, if one tries to generalize the ideas described in the previous Subsection for the construction of $\mathcal{D}_n^+$ to the infinite-order case, the resulting replacement system is expanding and has a limit space that is not a dendrite, and its rearrangement group has already been considered in literature:
it is the Airplane replacement system, depicted in \cref{fig:rep:sys:A}, and the rearrangement group $T_A$ was studied in \cite{Airplane}.

Differently from $\mathcal{D}_n^+$, the Airplane replacement system is expanding, and in particular red cells (cells were defined at \cpageref{sub:rearrangements}) do not collapse to branch points, but instead they end up forming infinitely many circles in the limit space, meaning that the limit space is certainly not a dendrite.
However, if one takes the quotient of the Airplane limit space by the equivalence relation that relates two points whenever they belong to the same circle of the Airplane (i.e., if they are represented by sequences $x \alpha$ and $x \beta$, where $x$ is any finite word ending with $b_2$ or $b_3$ and $\alpha$ and $\beta$ are infinite sequences in the alphabet $\{ r_1, r_2 \}$), then the resulting space is the infinite-order Wa\.zewski Dendrite $D_\infty$.
From this perspective, one can think of the Airplane rearrangement group $T_A$ as if it were some sort of a orientation-preserving dendrite rearrangement group $G_\infty^+$, because the elements of $T_A$ act by self-permutations on the set of these circles, so a faithful action of $T_A$ on $D_\infty$ is naturally defined.
We will be considering this action in the remainder of this section, but it is important to keep in mind that this is not the canonical action of a rearrangement group on its limit space.

Under this identification, the branch points of $D_\infty$ correspond precisely to the circles of the Airplane, and each branch at a fixed branch point corresponds to a dyadic point on the circle (see \cite{Airplane} for more detail).
The main distinction between a dendrite rearrangement group $G_n$ and $T_A$ thus consists in the fact that the group of permutations of branches around a branch point is Thompson's group $T$ instead of the finite permutation group $\mathrm{Sym}(n)$.

Thanks to the transitivity properties of $T_A$ proved in \cite{Airplane}, it is not hard to show that an analog of \cref{prop:trans} holds for $T_A$, as sketched below.
The construction of the trees $T(\mathcal{F})$ described at \cpageref{txt:trees} clearly works even if countably many branches are attached to each branch point.
We equip each tree $T(\mathcal{F})$ with the rotation system naturally prompted by the circular ordering of branches of $D_\infty$ induced by the circular ordering of the dyadic points on each circle of the Airplane.
(A \textbf{rotation system} on a graph is simply an assignment of a circular order to the edges incident on each vertex.)

\begin{proposition}
\label{prop:Airplane:trans}
Given two finite subsets $\mathcal{F}_1$ and $\mathcal{F}_2$ of $\mathrm{Br}$, any graph isomorphism between $T(\mathcal{F}_1)$ and $T(\mathcal{F}_2)$ that is compatible with their rotation systems can be extended to an element of $T_A$.
\end{proposition}

\begin{sketch}
By induction on the number $m$ of vertices of $T(\mathcal{F}_1)$ and $T(\mathcal{F}_2)$, the base case is single transitivity, which descends from \cite[Lemma 6.1]{Airplane}.
The inductive step is proved as in \cref{prop:trans}:
for $m+1$, exclude a leaf of $T(\mathcal{F}_1)$ and use the induction hypothesis to map the other $m$ vertices where they need to be mapped, then use the fact that Thompson's group $T$ acts $2$-transitively on the set of dyadic points of $S^1$ to find an element that fixes the aforementioned $m$ vertices and moves the previously excluded leaf to wherever it needs to be moved.
\end{sketch}

Consider the group $\mathbb{H}_\infty^+$ of all orientation-preserving homeomorphisms of $D_\infty$, by which we mean the group of those homeomorphisms of $D_\infty$ that preserve a circular order of branches at each branch point.
If we equip $\mathbb{H}_\infty^+$ with the subspace topology inherited from $\mathbb{H}_\infty$ then, as was done in \cref{sub:dns} for the density of $G_n$ in $\mathbb{H}_n$, we can prove that $T_A$ is dense in $\mathbb{H}_\infty^+$ by showing that every orientation-preserving homeomorphism can be ``approximated'' by elements of $T_A$, i.e.:

\begin{claim}
Let $\phi \in \mathbb{H}_\infty^+$.
For every $k \geq 1$ and for each $p_1, \dots, p_k \in \mathrm{Br}$ there exists a rearrangement $g_k \in T_A$ such that $g_k(p_i) = \phi(p_i)$ for all $i=1,\dots,k$.
\end{claim}

Using the previous \cref{prop:Airplane:trans}, the proof of this Claim is precisely the same as that of \cref{clm:dns}, so we have:

\begin{theorem}
\label{thm:dns:A}
The Airplane rearrangement group $T_A$ is dense in the group $\mathbb{H}_\infty^+$ of all orientation-preserving homeomorphisms of $D_\infty$.
\end{theorem}

\begin{remark}
Theorem 1 of \cite{BasilicaDense} shows that the Basilica rearrangement group $T_B$ is dense in the group $\mathrm{Aut}_+(T(\mathcal{B}))$ of orientation-preserving automorphisms of the (countably) infinite-degree regular tree.
Building on this result, the final remarks of \cite{BasilicaDense} mention that the Airplane rearrangement group $T_A$ is likely dense in the group of the orientation-preserving (in the same sense that we used elsewhere) homeomorphisms of the Tits-Bruhat $\mathbb{R}$-tree of the field of formal Laurent series $\mathbb{Q}\llbracket t, t^{1/2}, t^{1/4}, \dots \rrbracket$, which we will denote here by $\mathbb{T}$.
This real tree, which is not compact, is homeomorphic to the infinite-degree Wa\.zewski dendrite minus its endpoints.
Indeed, one can define a natural totally bounded metric on $\mathbb{T}$ in order to build its completion $\overline{\mathbb{T}}$, which is compact and is thus homeomorphic to the dendrite $D_\infty$.
Since the action of a homeomorphism of $\mathbb{T}$ entirely determines the action on the set $\overline{\mathbb{T}} \setminus \mathbb{T}$ of endpoints, the groups of (orientation-preserving) homeomorphisms of $D_\infty$ and $\overline{\mathbb{T}}$ are the same.
Hence, \cref{thm:dns:A} supports the claim that $T_A$ is dense in the group of the orientation-preserving homeomorphisms of the Tits-Bruhat $\mathbb{R}$-tree of the field of formal Laurent series $\mathbb{Q} \llbracket t, t^{1/2}, t^{1/4}, \dots \rrbracket$.
\end{remark}

\begin{remark}
As a final note about the infinite-order Wa\.zewski dendrite, we mention that it might be possible to build a (non orientation-preserving) rearrangement group for $D_\infty$ by either allowing the base and replacement graphs of a replacement system to be infinite, or maybe by building a direct limit from the inclusions $G_n \leq G_{n+1}$ (\cref{prop:inclusion}).
If one allows any graph isomorphism, the group built with the first method is not countable, as it contains a copy of $\mathrm{Sym}(\infty)$;
the second method instead seems to construct a ``finitary'' countable subgroup of the previous one that embeds into Thompson's group $V$.
We leave the exploration of these generalizations for future research.
\end{remark}

\subsection{Wa\.zewski Dendrite with Multiple Orders}
\label{sub:gen:den}

The final possible generalization that we mention is about the so-called \textbf{generalized Wa\.zewski dendrites}:
for each finite subset $S \subset \mathbb{N}_{\geq3}$, there exists a unique dendrite $D_S$ whose points have orders that belongs to $S \cup \{1, 2\}$ and such that every arc of $D_S$ contains points of every order in $S$.
Section 6 of \cite{DM18} (the entirety of which has been an important reference throughout this work) is about the full homeomorphism groups of generalized Wa\.zewski dendrites.

Given a finite $S \subset \mathbb{N}_{\geq3}$, we can build a (polychormatic) replacement system in the following way.
First, fix an ordering $s_1, \dots, s_k$ of the elements of $S$.
\begin{itemize}
    \item The set of colors is $\{1, \dots, k\}$.
    \item The $i$-th replacement graph is a tree consisting of a vertex of degree $s_i$ that is the origin of $s_i$ edges colored by $s_{i+1}$ terminating at $s_i$ distinct leaves;
    the initial and terminal vertices $\iota$ and $\tau$ are two distinct leaves.
    \item The base graph is the same as the first replacement graph (in truth, any replacement graph would work).
\end{itemize}
\cref{fig:rep:sys:col} depicts an example of such a replacement system, where $S = \{3, 4, 6\}$ and the chosen ordering is $4 \to 3 \to 6$.

The fact that the limit space is $D_S$ can be proved immediately as done at \cpageref{txt:rep:sys:den} for $\mathcal{D}_n$, so in particular the chosen ordering on $S$ does not change the limit space by \cite[Theorem 6.2]{selfhomeomorphisms}.
It is not clear, however, whether the rearrangement group is affected by this choice.
Then in order to keep track of this order, we denote by $G_{s_1, \dots, s_k}$ the corresponding rearrangement group.

\begin{figure}
\centering
\begin{tikzpicture}
    \node at (0,1.6) {$\Gamma$};
    \node[vertex] (l) at (-1,0) {};
    \node[vertex] (r) at (1,0) {};
    \node[vertex] (c) at (0,0) {};
    \node[vertex] (ct) at (0,1) {};
    \node[vertex] (cb) at (0,-1) {};
    \draw[edge,red] (c) to (l);
    \draw[edge,red] (c) to (r);
    \draw[edge,red] (c) to (ct);
    \draw[edge,red] (c) to (cb);
    \begin{scope}[xshift=3.333cm]
    \node at (0,1.6) {$R_{\textcolor{blue}{4}}$};
    \node[vertex] (l) at (-1,0) {}; \draw (-1,0) node[above]{$\iota$};
    \node[vertex] (r) at (1,0) {}; \draw (1,0) node[above]{$\tau$};
    \node[vertex] (c) at (0,0) {};
    \node[vertex] (ct) at (0,1) {};
    \node[vertex] (cb) at (0,-1) {};
    \draw[edge,red] (c) to (l);
    \draw[edge,red] (c) to (r);
    \draw[edge,red] (c) to (ct);
    \draw[edge,red] (c) to (cb);
    \end{scope}
    \begin{scope}[xshift=6.667cm]
    \node at (0,1.6) {$R_{\textcolor{red}{3}}$};
    \node[vertex] (l) at (-1,0) {}; \draw (-1,0) node[above]{$\iota$};
    \node[vertex] (r) at (1,0) {}; \draw (1,0) node[above]{$\tau$};
    \node[vertex] (c) at (0,0) {};
    \node[vertex] (ct) at (0,1) {};
    \draw[edge,Green] (c) to (l);
    \draw[edge,Green] (c) to (r);
    \draw[edge,Green] (c) to (ct);
    \end{scope}
    \begin{scope}[xshift=9.5cm]
    \node at (0,1.6) {$R_{\textcolor{Green}{6}}$};
    \node[vertex] (l) at (-1,0) {}; \draw (-1,0) node[above]{$\iota$};
    \node[vertex] (r) at (1,0) {}; \draw (1,0) node[above]{$\tau$};
    \node[vertex] (c) at (0,0) {};
    \node[vertex] (v1) at (60:1) {};
    \node[vertex] (v2) at (120:1) {};
    \node[vertex] (v3) at (240:1) {};
    \node[vertex] (v4) at (300:1) {};
    \draw[edge,blue] (c) to (l);
    \draw[edge,blue] (c) to (r);
    \draw[edge,blue] (c) to (v1);
    \draw[edge,blue] (c) to (v2);
    \draw[edge,blue] (c) to (v3);
    \draw[edge,blue] (c) to (v4);
    \end{scope}
\end{tikzpicture}
\caption{A replacement system for the generalized Wa\.zewski dendrite $D_{\{3,4,6\}}$ with the ordering $4 \to 3 \to 6$ on $S$.}
\label{fig:rep:sys:col}
\end{figure}
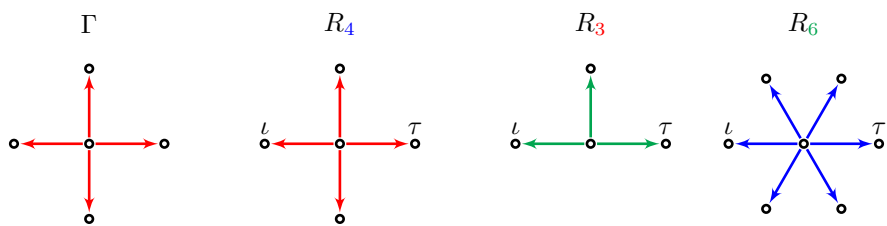

It is plausible that $G_{s_1, \dots, s_k}$ is generated by a copy for each of the finite groups $\mathrm{Sym}(s_i)$ along with the subgroup of Thompson's group $F$ that preserves the coloration of the edges.
If this subgroup is transitive enough, then $G_{s_1, \dots, s_k}$ is likely to be dense in $\mathrm{Homeo}(D_S)$.
The study of these groups is left for future research.

Finally, one could ``mix'' the various generalizations proposed throughout this Section.
For example, there may be orientation-preserving rearrangement groups of generalized Wa\.zewski dendrites where the Thompson's group acting on each EE-arc is a color-preserving subgroup of some Higman-Thompson group $F_m$.
This too is left for future research.


\section*{Acknowledgements}

The author would like to thank Jim Belk, Ilaria Castellano, Indira Chatterji, Bruno Duchesne, Bianca Marchionna, Francesco Matucci, Yury Neretin, Feyishayo Olukoya, Davide Perego, Stefan Witzel and the anonymous referee for useful advices, discussions, comments and corrections.
The author is also grateful to Jim Belk and Bradley Forrest for kindly providing their images of the Airplane and Vicsek limit spaces depicted in \cref{fig:airplane:limit:space,fig:T(F)}, from their work \cite{belk2016rearrangement}.


\emergencystretch=2em
\printbibliography[heading=bibintoc]

\end{document}